\documentclass[twoside,english]{amsart}
\usepackage[T1]{fontenc}
\usepackage[latin9]{inputenc}
\usepackage[a4paper]{geometry}
\usepackage{babel}
\usepackage{amstext}
\usepackage{amsthm}
\usepackage{amssymb}
\usepackage{setspace}
\usepackage[colorlinks=true,linkcolor=black,anchorcolor=black,citecolor=black,filecolor=black,menucolor=black,runcolor=black,urlcolor=black]{hyperref}

\makeatletter
\numberwithin{equation}{section}
\numberwithin{figure}{section}
\theoremstyle{plain}
\newtheorem{thm}{\protect\theoremname}
\theoremstyle{plain}
\newtheorem{lem}[thm]{\protect\lemmaname}
\theoremstyle{plain}
\newtheorem{cor}[thm]{\protect\corollaryname}

\makeatother

\providecommand{\corollaryname}{Corollary}
\providecommand{\lemmaname}{Lemma}
\providecommand{\theoremname}{Theorem}

\begin{document}
\title{Coercive Inequalities in Higher-Dimensional Anisotropic Heisenberg
Group}
\author{E. Bou Dagher}
\address{Esther Bou Dagher: \newline Department of Mathematics \newline Imperial College London \newline 180 Queen's Gate, London SW7 2AZ \newline United Kingdom}
\email{esther.bou-dagher17@imperial.ac.uk}
\author{B. Zegarli\'{n}ski}
\address{Bogus\l{}aw Zegarli\'{n}ski: \newline Department of Mathematics \newline Imperial College London \newline 180 Queen's Gate, London SW7 2AZ \newline United Kingdom}
\email{b.zegarlinski@imperial.ac.uk}

\providecommand{\keywords}[1]{\textbf{\textit{Keywords---}} #1}

\begin{abstract}
In the setting of higher-dimensional anisotropic Heisenberg group,
we compute the fundamental solution for the sub-Laplacian, and we
prove Poincaré and $\beta-$Logarithmic Sobolev inequalities for measures
as a function of this fundamental solution. 

\tableofcontents{}

\end{abstract}


\maketitle

\section{Introduction}

In 1975, G.B. Folland showed in \cite{key-7} that on a Carnot group
$\mathbb{G},$ the sub-Laplacian $\triangle:={\displaystyle \sum_{i=1}^{n}X_{i}^{2}}$
admits a unique fundamental solution $N^{2-Q},$ i.e.
\[
\triangle N^{2-Q}=\delta,
\]
where $\delta$ is the delta-distribution at the unit element of $\mathbb{G},$
$X_{1},...,X_{n}$ are the Jacobian generators of $\mathbb{G},$ $Q$
is the homogeneous dimension, and $N$ is a homogeneous norm on $\mathbb{G}$. 

In \cite{key-3}, Z. Balogh and J. Tyson introduced the concept of
polarizable Carnot groups defined by the condition that $N$ is $\infty-$harmonic
in $\mathbb{G\backslash}{\{0\}}$, i.e. for ${\displaystyle \triangledown:=}(X_{i})_{1\leq i\leq n},$
\begin{equation}
\triangle_{\infty}N:=\frac{1}{2}<\triangledown\left(|\triangledown N|^{2}\right),\triangledown N>=0\;\;\;\;\;\;in\;\;\mathbb{G\backslash}\{0\}.\label{eq:1}
\end{equation}
They have shown that using the $\infty-$harmonicity of $N$ one can
provide a procedure to construct polar coordinates of special type
where the curves passing through the points on the unit sphere ${\{N=1\}}$
are horizontal.

Moreover, they showed in \cite{key-3} that the fundamental solution
of the $p-$sub-Laplacian can be expressed as the fundamental solution
$N$ of the sub-Laplacian, proved capacity formulas, and produced
sharp constants for the Moser-Trudinger inequality (which was established
by L. Saloff-Coste in \cite{key-8} in the setting of Carnot groups
but without sharp constants). In settings related to polarizable Carnot
groups, many authors showed Hardy-type inequalities as a function
of $N$ (see \cite{key-10,key-13,key-14,key-15,key-16,key-9}), proved
Rellich-type inequalities (see \cite{key-15,key-17}), and studied
Fuglede's $p-$module of system of measures \cite{key-19}. 

For the time being, there is no a classification of polarizable Carnot
groups, and the only examples till now are Euclidean spaces and Heisenberg-type
groups. In addition, the concept of a polarizable Carnot group is
a delicate one in the sense that under a small pertubation of the
Lie algebra, the group is no longer polarizable. Z. Balogh and J.
Tyson provided in \cite{key-3} the anisotropic Heisenberg group in
$\mathbb{R}^{4}$ as a counterexample with the following generators
of the Lie algebra: $X=\frac{\partial}{\partial x}+2ay\cdot\frac{\partial}{\partial t},$
$Y=\frac{\partial}{\partial y}-2ax\cdot\frac{\partial}{\partial t},$
$Z=\frac{\partial}{\partial z}-2w\cdot\frac{\partial}{\partial t},$
and $W=\frac{\partial}{\partial w}-2z\cdot\frac{\partial}{\partial t},$
where $a=\frac{1}{2}.$ (Note that if $a=1,$ we have the polarizable
Heisenberg group.) To show (\ref{eq:1}) does not hold true for the
anisotropic Heisenberg group, they computed explicitly the fundamental
solution of the sub-Laplacian using R. Beals, B. Gaveau, and P. Greiner's
\cite{key-20} explicit intergal representation for the fundamental
solution in the setting of general step-two Carnot groups.

Recently, T. Bieske \cite{key-21} revisited this counterexample and
proved that under a change of the inner product imposed on the vectors
in the Lie Algebra, which now requires the generators to be orthogonal
instead of orthonormal, the anisotropic Heisenberg group is turned
into a group of Heisenberg-type i.e. it is now polarizable!

The goal of this paper is to study coercive inequalities such as the
$q-$Poincaré inequality and the $\beta-$Logarithmic Sobolev inequality
in the setting of the anisotropic Heisenberg group with respect to
measures as a function of the explicit fundamental solution. We will
first show the computations, that were partially omitted in \cite{key-3}
in the $\mathbb{R}^{5}$ setting, and use that to get an explicit
fundamental solution for higher dimensions.

\begin{onehalfspace}
In the setting of nilpotent Lie groups, heat kernel estimates have
been used to get coercive inequalities \cite{key-22,key-23,key-25,key-26,key-27,key-28,key-29,key-30,key-31,key-32}.
In our setting, we do not use heat kernel estimates; instead, we study
coercive inequalities involving sub-gradients and probability measures
depending on the group. An approach to study such problems was pioneered
in \cite{key-33}. It was later used by J. Inglis to get Poincaré
inequality in the setting of the Heisenberg-type group with measure
as a function of Kaplan distance \cite{key-34}, by M. Chatzakou et
al. to get Poincaré inequality in the setting of the Engel-type group
with a measure as a function of some quasi-homegenous norm \cite{key-35},
and by the authors of this paper in \cite{key-36} to get $q-$Poincaré
inequality and the $\phi-$Logarithmic Sobolev in the setting of step-two
carnot groups with measures as function of a generalised Kaplan norm
\cite{key-36}.
\end{onehalfspace}

The method of \cite{key-33} to get coercive inequalities relies on
the study of so-called U-bounds, i.e. bounds of the following type
\begin{equation}
\int|f|^{q}Ud\mu\leq C\int|\triangledown f|^{q}d\mu+D\int|f|^{q}d\mu,\label{eq:ubound1}
\end{equation}
with functions $U$ possessing suitable growth properties at infinity.
Our key result in this paper is obtaining (in section 3) for a probability
measure $d\mu=\frac{e^{-g(N)}}{Z}d\lambda,$ defined with $g(N)$
satisfying suitable growth conditions, the following U-Bound in the
setting of the higher-dimensional anisotropic Heisenberg group:
\begin{equation}
\int\frac{g'\left(N\right)}{N^{2}}\vert f\vert^{q}d\mu\leq C\int\vert\triangledown f\vert^{q}d\mu+D\int\vert f\vert^{q}d\mu,\label{eq:u}
\end{equation}
where ${\displaystyle \triangledown:=}(X_{i})_{1\leq i\leq n},$ with
constants $C,D\in(0,\infty)$ independent of the function $f$ for
which the right-hand side is well defined. This U-bound is used (in
section 4) to get a q-Poincaré inequality and a $\beta-$Logarithmic
Sobolev inequality for $q\geq2$ . We expect that our results can
be used to extend those coercive inequalities to an infinite dimensional
setting, which is of interest. (See also works: \cite{key-42,key-41,key-40,key-39,key-38,key-37}.)
In the second section, we will start by extending Z. Balogh and J.
Tyson's anisotropic Heisenberg group in $\mathbb{R}^{5}$ \cite{key-3}
to a higher-dimensional anisotropic Heisenberg group in $\mathbb{R}^{2n+1},$
and use R. Beals, B. Gaveau, and P. Greiner's \cite{key-20} explicit
intergal representation to compute the fundamental solution. We also
compute bounds for $|\triangledown N|$ and $x\cdot\triangledown N$
(section 2), which are essential to get the U-Bound (\ref{eq:u})
(section 3). We remark that in our setting, unlike in the case studied
in \cite{key-36}, $x\cdot\triangledown N$ can be negative, and we
will need the dimension $n>5$ to take care of the negative term.
For $n\leq5,$ some other method is yet to be explored to get U-Bounds.
Finally, in the fourth section, we apply the U-Bound to get coercive
inequalities like the q-Poincaré inequality and the $\beta-$Logarithmic
Sobolev inequality for $q\geq2$ . 

\section{The Fundamental Solution}

Consider a generalisation to the anisotropic Heisenberg group $\mathbb{H}_{2n}(\frac{1}{2},1),$
as introduced in \cite{key-3}, on $\mathbb{R}^{2n+1}$ with dilation
$\delta_{\lambda}(x_{1},x_{2},...,x_{2n},t)=(\lambda x_{1},\lambda x_{2},...,\lambda x_{2n},\lambda^{2}t)$
and the composition law
\[
(x_{1},x_{2},...,x_{2n},t)\circ(\eta_{1},\eta_{2},...,\eta_{2n},\tau)
\]
\[
=\left(x_{1}+\eta_{1},x_{2}+\eta_{2},...,x_{2n}+\eta_{2n},t+\tau+\frac{x_{1}\eta_{n+1}}{2}-\frac{x_{n+1}\eta_{1}}{2}+\sum_{j=2}^{n}(x_{j}\eta_{j+n}-\eta_{j}x_{j+n})\right).
\]
$\mathbb{H}_{2n}(\frac{1}{2},1)$ is a homogeneous Carnot group of
step two with generators
\[
X_{j}=\begin{cases}
\partial_{x_{1}}-\frac{x_{n+1}}{2}\partial_{t} & j=1\\
\partial_{x_{n+1}}+\frac{x_{1}}{2}\partial_{t} & j=n+1\\
\partial_{x_{j}}-x_{j+n}\partial_{t} & j=2,3,..,n\\
\partial_{x_{j}}+x_{j-n}\partial_{t} & j=n+2,n+3,..,2n.
\end{cases}
\]

\begin{thm}
The fundamental solution for the group $\mathbb{H}_{2n}(\frac{1}{2},1)$
is given by the following homogeneous norm
\begin{equation}
N(x,t)=\frac{\left(B^{2}+t^{2}\right)^{\frac{1}{4n}}\left(AB+t^{2}+A\sqrt{B^{2}+t^{2}}\right)^{\frac{1}{2}-\frac{1}{4n}}}{\left(B+\sqrt{B^{2}+t^{2}}\right)^{\frac{1}{2}}}\label{eq:9-2}
\end{equation}
where $A={\displaystyle \frac{x_{1}^{2}}{2}}+{\displaystyle \frac{x_{n+1}^{2}}{2}}+{\displaystyle \frac{1}{2}}\sum_{j=1,j\not=n+1}^{2n}x_{j}^{2}$
and $B={\displaystyle \frac{x_{1}^{2}}{4}+\frac{x_{n+1}^{2}}{4}}+{\displaystyle \frac{1}{2}}\sum_{j=1,j\not=n+1}^{2n}x_{j}^{2}.$
\end{thm}

\textbf{Remark:} Note that for $n=2$, we get back the homogeneous
norm associated to the fundamental solution as calculated in \cite{key-3}.

For what follows, we denote ${\displaystyle |x|=\left(\sum_{j=1}^{2n}x_{j}^{2}\right)^{\frac{1}{2}}}$to
be the Euclidean norm. The following lemma is crucial in obtaining
a U-Bound in section 3.
\begin{lem}
The homogeneous norm $N$ on $\mathbb{H}_{2n}(\frac{1}{2},1)$ satisfies
\begin{equation}
x\cdot\triangledown N\geq-\frac{|x|^{2}}{4nN},\label{eq:14-1}
\end{equation}
\begin{equation}
|\triangledown N|^{2}\geq\frac{|x|^{2}}{2^{5+\frac{2}{n}}N^{2}},\label{eq:20-2}
\end{equation}
and
\begin{equation}
|\triangledown N|^{2}\leq\frac{(2n+1)^{2}|x|^{2}}{2^{3}n^{2}N^{2}}.\label{eq:24-2}
\end{equation}
\end{lem}

In the rest of this section we provide a proof of the above lemmata
which involves lengthy calculations based on general formula for Green
functions for type-2 Carnot groups provided in \cite{key-20}. The
reader interested in applications to coercive inequalities is invited
to jump directly to section 3.\\

\subsection{Derivation of the Formula for Homogeneous Norm $N$}

\hspace{11cm}

To compute explicitly the fundamental solution, we now use R. Beals,
B. Gaveau, and P. Greiner's Theorem 2 of \cite{key-20}. The fundamental
solution for the sub-laplacian $\triangle={\displaystyle \sum_{j=1}^{2n}X_{j}^{2}}$
on $\mathbb{G}$ with singularity at $0$ is:
\begin{equation}
u(x,t)=\frac{\Gamma(\frac{Q}{2}-1)}{2(2\pi)^{\frac{Q}{2}}}\int_{\mathbb{R}}\frac{V(\tau)}{f(x,t,\tau)^{\frac{Q}{2}-1}}d\tau,\label{eq:1-1}
\end{equation}
for $(x,t)\in\mathbb{G},\;\;x\not=0.$ Where we have $\;Q=2n+2$ the
homogeneous dimension, $V:\mathbb{R}\rightarrow\mathbb{R}$ such that
\[
V(\tau)=\prod_{j=1}^{2n}w_{j}(\tau)^{\frac{1}{2}}csch(w_{j}(\tau))^{\frac{1}{2}},
\]
and $f:\mathbb{G}\times\mathbb{R}\rightarrow\mathbb{C}$ such that
\[
f(x,t,\tau)=\frac{1}{2}\sum_{j=1}^{2n}w_{j}(\tau)coth(w_{j}(\tau))|x\cdot e_{j}(\tau)|^{2}-\sqrt{-1}t\tau.
\]
Here, $w_{j}(\tau)$ are the eigenvalues and $e_{j}(\tau)$ are the
corresponding normalised eigenvectors of the $2n\times2n$ matrix
$\Omega(\tau)=-\sqrt{-1}\tau M,$ where

$M_{kl}=\begin{cases}
-\frac{1}{2} & k=1,\ \ l=n+1\\
\frac{1}{2} & k=n+1,\ \ l=1\\
-1 & k=2,3,...,n,\ \ l=k+n\\
1 & k=n+2,n+3,...,2n,\ \ l=k-n\\
0 & otherwise
\end{cases}$\\
The eigenvalues are: $w_{1}(\tau)=-{\displaystyle \frac{\tau}{2},}$
$w_{n+1}(\tau)={\displaystyle \frac{\tau}{2},}$ $w_{j}(\tau)={\displaystyle -\tau,}$
for $j=2,3,..,n,$ and $w_{j}(\tau)={\displaystyle \tau,}$ for $j=n+2,n+3,...,2n.$
The corresponding normalised eigenvectors are:
\[
(e_{j}(\tau))_{i}=\begin{cases}
-\frac{\sqrt{-1}}{\sqrt{2}} & i=j\\
\frac{1}{\sqrt{2}} & i=j+n\\
0 & otherwise,
\end{cases}
\]
for $j=1,2,...,n,$ and
\[
(e_{j}(\tau))_{i}=\begin{cases}
\frac{\sqrt{-1}}{\sqrt{2}} & i=j-n\\
\frac{1}{\sqrt{2}} & i=j\\
0 & otherwise,
\end{cases}
\]
for $j=n+1,n+2,...,2n.$ Thus, 

\begin{equation}
V(\tau)=\prod_{j=1}^{2n}w_{j}(\tau)^{\frac{1}{2}}csch(w_{j}(\tau))^{\frac{1}{2}}=(\frac{\tau^{2}}{4}csch\left(\frac{\tau}{2}\right)^{2}\tau^{2n-2}csch(\tau)^{2n-2})^{\frac{1}{2}}=\frac{\tau^{n}}{2}csch\left(\frac{\tau}{2}\right)csch(\tau)^{n-1},\label{eq:2}
\end{equation}
\[
\begin{array}{cl}
f(x,t,\tau) & {\displaystyle =\frac{1}{2}\sum_{j=1}^{2n}w_{j}(\tau)coth(w_{j}(\tau))|x\cdot e_{j}(\tau)|^{2}-\sqrt{-1}t\tau}\\
\\
 & {\displaystyle =\frac{1}{2}w_{1}(\tau)coth(w_{1}(\tau))|x\cdot e_{1}(\tau)|^{2}+\frac{1}{2}w_{n+1}(\tau)coth(w_{n+1}(\tau))|x\cdot e_{n+1}(\tau)|^{2}}\\
\\
 & +\frac{1}{2}\sum_{j=1,j\not=n+1}^{2n}w_{j}(\tau)coth(w_{j}(\tau))|x\cdot e_{j}(\tau)|^{2}-\sqrt{-1}t\tau\\
\\
 & {\displaystyle =-{\displaystyle \frac{\tau}{4}coth\left(-{\displaystyle \frac{\tau}{2}}\right)\left|-\frac{\sqrt{-1}}{\sqrt{2}}x_{1}+\frac{x_{n+1}}{\sqrt{2}}\right|^{2}+{\displaystyle \frac{\tau}{4}coth\left({\displaystyle \frac{\tau}{2}}\right)\left|\frac{\sqrt{-1}}{\sqrt{2}}x_{1}+\frac{x_{n+1}}{\sqrt{2}}\right|^{2}}}}\\
\\
 & +{\displaystyle \frac{1}{2}}{\displaystyle \sum_{j=2}^{n}-\tau coth(-\tau)\left|-\frac{\sqrt{-1}}{\sqrt{2}}x_{j}+\frac{x_{j+n}}{\sqrt{2}}\right|^{2}+\frac{1}{2}\sum_{j=n+2}^{2n}\tau coth(\tau)\left|\frac{\sqrt{-1}}{\sqrt{2}}x_{j}+\frac{x_{j+n}}{\sqrt{2}}\right|^{2}-\sqrt{-1}t\tau}\\
\\
 & {\displaystyle =\frac{\tau}{2}coth({\displaystyle \frac{\tau}{2}})(\frac{x_{1}^{2}}{2}+\frac{x_{n+1}^{2}}{2})+\frac{\tau}{2}coth({\displaystyle \tau})(\sum_{j=1,j\not=n+1}^{2n}x_{j}^{2})-\sqrt{-1}t\tau.}
\end{array}
\]
Letting $A={\displaystyle \frac{x_{1}^{2}}{2}}+{\displaystyle \frac{x_{n+1}^{2}}{2}}+{\displaystyle \frac{1}{2}}\sum_{j=1,j\not=n+1}^{2n}x_{j}^{2}$
and $B={\displaystyle \frac{x_{1}^{2}}{4}+\frac{x_{n+1}^{2}}{4}}+{\displaystyle \frac{1}{2}}\sum_{j=1,j\not=n+1}^{2n}x_{j}^{2},$
we obtain:
\begin{equation}
f(x,t,\tau)=\tau coth\left({\displaystyle \frac{\tau}{2}}\right)(A-B)+\tau coth({\displaystyle \tau})(2B-A)-\sqrt{-1}t\tau.\label{eq:3}
\end{equation}

Replacing (\ref{eq:2}) and (\ref{eq:3}) in (\ref{eq:1-1}), we get:
\[
u(x,t)=\frac{\Gamma(n)}{2(2\pi)^{n+1}}\int_{\mathbb{R}}\frac{V(\tau)}{f(x,t,\tau)^{n}}d\tau
\]
\[
=\frac{\Gamma(n)}{2(2\pi)^{n+1}}\int_{\mathbb{R}}\frac{\frac{\tau^{n}}{2}csch(\frac{\tau}{2})csch(\tau)^{n-1}}{\left(\tau coth\left({\displaystyle \frac{\tau}{2}}\right)(A-B)+\tau coth({\displaystyle \tau})(2B-A)-\sqrt{-1}t\tau\right)^{n}}d\tau
\]
\[
=\frac{\Gamma(n)}{(2\pi)^{n+1}}\int_{0}^{\infty}Re\left\{ \frac{\frac{\tau^{n}}{2}csch(\frac{\tau}{2})csch(\tau)^{n-1}}{\left(\tau coth\left({\displaystyle \frac{\tau}{2}}\right)(A-B)+\tau coth({\displaystyle \tau})(2B-A)-\sqrt{-1}t\tau\right)^{n}}\right\} d\tau
\]
\[
=\frac{\Gamma(n)}{2(2\pi)^{n+1}}\int_{0}^{\infty}Re\left\{ \frac{csch(\frac{\tau}{2})csch(\tau)^{n-1}coth\left(\frac{\tau}{2}\right)^{n}}{coth\left(\frac{\tau}{2}\right)^{n}\left(coth\left({\displaystyle \frac{\tau}{2}}\right)(A-B)+coth({\displaystyle \tau})(2B-A)-\sqrt{-1}t\right)^{n}}\right\} d\tau
\]

using the trigonometric identity ${\displaystyle coth(\tau)=\frac{1}{2coth\left(\frac{\tau}{2}\right)}+\frac{coth\left(\frac{\tau}{2}\right)}{2},}$

\[
=\frac{\Gamma(n)}{2(2\pi)^{n+1}}\int_{0}^{\infty}Re\left\{ \frac{csch(\frac{\tau}{2})csch(\tau)^{n-1}coth\left(\frac{\tau}{2}\right)^{n}}{\left(\frac{A}{2}coth\left({\displaystyle \frac{\tau}{2}}\right)^{2}+B-\frac{A}{2}-\sqrt{-1}t\ coth\left({\displaystyle \frac{\tau}{2}}\right)\right)^{n}}\right\} d\tau
\]
using the trigonometric identities $csch(\tau)={\displaystyle \frac{(csch\left(\frac{\tau}{2}\right))^{2}}{2coth\left(\frac{\tau}{2}\right)}}$
and ${\displaystyle coth\left({\displaystyle \frac{\tau}{2}}\right)^{2}=1+csch\left({\displaystyle \frac{\tau}{2}}\right)^{2},}$
\[
=\frac{\Gamma(n)}{(2\pi)^{n+1}}\int_{0}^{\infty}Re\left\{ \frac{csch(\frac{\tau}{2})csch\left(\frac{\tau}{2}\right)^{2n-2}coth\left(\frac{\tau}{2}\right)}{\left(Acsch\left({\displaystyle \frac{\tau}{2}}\right)^{2}+2B-2\sqrt{-1}t\ \sqrt{1+csch\left({\displaystyle \frac{\tau}{2}}\right)^{2}}\right)^{n}}\right\} d\tau
\]
letting $u=csch\left({\displaystyle \frac{\tau}{2}}\right),$ we have
${\displaystyle du=-\frac{1}{2}coth\left({\displaystyle \frac{\tau}{2}}\right)}csch\left({\displaystyle \frac{\tau}{2}}\right)d\tau$
\[
=\frac{\Gamma(n)}{(2\pi)^{n}}Re\left\{ \int_{0}^{\infty}\frac{u^{2n-2}}{\left(Au^{2}+2B-2\sqrt{-1}t\ \sqrt{1+u^{2}}\right)^{n}}du\right\} 
\]
\[
=\frac{\Gamma(n)}{(2\pi)^{n}}Re\left\{ \frac{(-1)^{n-2}}{(n-1)!}\frac{\partial^{n-2}}{\partial A^{n-2}}\int_{0}^{\infty}\frac{u^{2}}{\left(Au^{2}+2B-2\sqrt{-1}t\ \sqrt{1+u^{2}}\right)^{2}}du\right\} 
\]
\begin{equation}
=\frac{(-1)^{n-2}}{(2\pi)^{n}}\frac{d^{n-2}}{dA^{n-2}}Re\left\{ \int_{0}^{\infty}\frac{u^{2}}{\left(Au^{2}+2B-2\sqrt{-1}t\ \sqrt{1+u^{2}}\right)^{2}}du\right\} .\label{eq:4}
\end{equation}
The next step is to compute
\[
I(A,B,t)=Re\left\{ \int_{0}^{\infty}\frac{u^{2}}{\left(Au^{2}+2B-2\sqrt{-1}t\ \sqrt{1+u^{2}}\right)^{2}}du\right\} 
\]
\[
=\int_{0}^{\infty}\frac{(Au^{2}+2B)^{2}u^{2}-4t^{2}(1+u^{2})u^{2}}{\left((Au^{2}+2B)^{2}+4t^{2}(1+u^{2})\right)^{2}}du.
\]
Write $I(A,B,t)=I_{1}(A,B,t)-I_{2}(A,B,t),$ where
\[
I_{1}(A,B,t)=\int_{0}^{\infty}\frac{u^{2}}{(Au^{2}+2B)^{2}+4t^{2}(1+u^{2})}du
\]
and
\[
I_{2}(A,B,t)=\int_{0}^{\infty}\frac{8t^{2}(1+u^{2})u^{2}}{\left((Au^{2}+2B)^{2}+4t^{2}(1+u^{2})\right)^{2}}du.
\]
Remark that
\[
I_{2}(A,B,t)=-t\partial_{t}I_{1}(A,B,t),
\]
so
\begin{equation}
I(A,B,t)=I_{1}(A,B,t)+t\partial_{t}I_{1}(A,B,t).\label{eq:5}
\end{equation}
We now compute $I_{1}(A,B,t)={\displaystyle \frac{1}{2}\int_{-\infty}^{\infty}\frac{x^{2}}{(Ax^{2}+2B)^{2}+4t^{2}(1+x^{2})}dx}$
using complex theory. Write denominator in the integrand as a polynomial
\begin{equation}
p(z)=(Az^{2}+2B)^{2}+4t^{2}(1+z^{2})=A^{2}z^{4}+(4AB+4t^{2})z^{2}+4B^{2}+4t^{2}.\label{eq:6}
\end{equation}
$p(z)$ has four complex roots. If $z$ is a root, then $-z$ is a
root. Since $p(z)$ has real coefficients, then $\bar{z}$ is a root,
and $-\bar{z}$ is the fourth root. Let $\alpha$ be a root in the
upper half plane, so $-\bar{\alpha}$ is also in the upper half plane.
We thus have
\[
p(z)=A^{2}(z-\alpha)(z+\alpha)(z-\bar{\alpha})(z+\bar{\alpha})=A^{2}(z^{4}-(\alpha^{2}+\bar{\alpha}^{2})z^{2}+|\alpha|^{4}).
\]
By identification with (\ref{eq:6}), we get that $\alpha^{2}+\bar{\alpha}^{2}=-{\displaystyle \frac{4AB+4t^{2}}{A^{2}}}$
and ${\displaystyle |\alpha|^{4}=\frac{4B^{2}+4t^{2}}{A^{2}}.}$ From
this, we can calculate ${\displaystyle Im(\alpha)=\frac{\alpha-\bar{\alpha}}{2i}.}$
In fact,
\begin{equation}
Im(\alpha)^{2}=-\frac{\alpha^{2}+\bar{\alpha}^{2}-2|\alpha|^{2}}{4}=\frac{AB+t^{2}+A\sqrt{B^{2}+t^{2}}}{A^{2}}.\label{eq:7}
\end{equation}
Now, we go back to computing the integral: Let ${\displaystyle f(z)=\frac{z^{2}}{(Az^{2}+2B)^{2}+4t^{2}(1+z^{2})}}$,
by the residue theorem:
\begin{equation}
I_{1}(A,B,t)=\pi i\left(Res(f,\alpha)+Res(f,-\bar{\alpha})\right).\label{eq:8}
\end{equation}
\[
Res(f,\alpha)=\underset{z\rightarrow\alpha}{lim}(z-\alpha)f(z)=\underset{z\rightarrow\alpha}{lim}(z-\alpha)\frac{z^{2}}{A^{2}(z-\alpha)(z+\alpha)(z-\bar{\alpha})(z+\bar{\alpha})}=\frac{\alpha}{2A^{2}(\alpha^{2}-\bar{\alpha}^{2})}
\]
and
\[
Res(f,-\bar{\alpha})=\underset{z\rightarrow-\bar{\alpha}}{lim}(z+\bar{\alpha})f(z)=\underset{z\rightarrow-\bar{\alpha}}{lim}(z+\bar{\alpha})\frac{z^{2}}{A^{2}(z-\alpha)(z+\alpha)(z-\bar{\alpha})(z+\bar{\alpha})}=\frac{\bar{\alpha}}{2A^{2}(\alpha^{2}-\bar{\alpha}^{2})}.
\]
Replacing in (\ref{eq:8}), we get:
\[
I_{1}(A,B,t)=\pi i\left(Res(f,\alpha)+Res(f,-\bar{\alpha})\right)
\]
\[
=\pi i\left(\frac{\alpha}{2A^{2}(\alpha^{2}-\bar{\alpha}^{2})}+\frac{\bar{\alpha}}{2A^{2}(\alpha^{2}-\bar{\alpha}^{2})}\right)
\]
\[
=\frac{2\pi i}{4A^{2}(\alpha-\bar{\alpha})}=\frac{\pi}{4A^{2}Im(\alpha)}
\]
using (\ref{eq:7})
\[
=\frac{\pi}{4A\left(AB+t^{2}+A\sqrt{B^{2}+t^{2}}\right)^{\frac{1}{2}}}.
\]
We can now compute $I(A,B,t)$ using (\ref{eq:5}):
\[
I(A,B,t)=I_{1}(A,B,t)+t\partial_{t}I_{1}(A,B,t)
\]
\[
=\frac{\pi}{4A\left(AB+t^{2}+A\sqrt{B^{2}+t^{2}}\right)^{\frac{1}{2}}}+t\partial_{t}\frac{\pi}{4A\left(AB+t^{2}+A\sqrt{B^{2}+t^{2}}\right)^{\frac{1}{2}}}
\]
\[
=\frac{\pi}{4A\left(AB+t^{2}+A\sqrt{B^{2}+t^{2}}\right)^{\frac{1}{2}}}-\frac{\pi(2t^{2}\sqrt{B^{2}+t^{2}}+At^{2})}{8A\sqrt{B^{2}+t^{2}}\left(AB+t^{2}+A\sqrt{B^{2}+t^{2}}\right)^{\frac{3}{2}}}
\]
\[
=\pi\frac{2\sqrt{B^{2}+t^{2}}\left(AB+t^{2}+A\sqrt{B^{2}+t^{2}}\right)-2t^{2}\sqrt{B^{2}+t^{2}}-At^{2}}{8A\sqrt{B^{2}+t^{2}}\left(AB+t^{2}+A\sqrt{B^{2}+t^{2}}\right)^{\frac{3}{2}}}
\]
\[
=\pi\frac{2B\sqrt{B^{2}+t^{2}}+2B^{2}+t^{2}}{8\sqrt{B^{2}+t^{2}}\left(AB+t^{2}+A\sqrt{B^{2}+t^{2}}\right)^{\frac{3}{2}}}
\]
\[
=\frac{\pi\left(B+\sqrt{B^{2}+t^{2}}\right)^{2}}{8\sqrt{B^{2}+t^{2}}\left(AB+t^{2}+A\sqrt{B^{2}+t^{2}}\right)^{\frac{3}{2}}}.
\]
Using (\ref{eq:4}),

\[
u(x,t)=\frac{(-1)^{n-2}}{(2\pi)^{n}}\frac{d^{n-2}}{dA^{n-2}}I(A,B,t)
\]

\[
=\frac{(-1)^{n-2}}{(2\pi)^{n}}\frac{d^{n-2}}{dA^{n-2}}\left(\frac{\pi\left(B+\sqrt{B^{2}+t^{2}}\right)^{2}}{8\sqrt{B^{2}+t^{2}}\left(AB+t^{2}+A\sqrt{B^{2}+t^{2}}\right)^{\frac{3}{2}}}\right)
\]

\[
=\frac{(-1)^{n-2}(-1)^{n-2}\prod_{k=3}^{n}(2k-3)}{(2\pi)^{n}2^{n-2}}\left(\frac{\pi\left(B+\sqrt{B^{2}+t^{2}}\right)^{2+n-2}}{8\sqrt{B^{2}+t^{2}}\left(AB+t^{2}+A\sqrt{B^{2}+t^{2}}\right)^{\frac{3}{2}+n-2}}\right)
\]

\[
=\frac{\prod_{k=3}^{n}(2k-3)}{2\pi^{n-1}2^{2n}}\left(\frac{\left(B+\sqrt{B^{2}+t^{2}}\right)^{n}}{\sqrt{B^{2}+t^{2}}\left(AB+t^{2}+A\sqrt{B^{2}+t^{2}}\right)^{n-\frac{1}{2}}}\right).
\]

The fundamental solution is given up to a constant multiple:
\[
N(x,t)=u(x,t)^{\frac{1}{2-Q}}=u(x,t)^{-\frac{1}{2n}}
\]
using the last equation,
\begin{equation}
N(x,t)=\frac{\left(B^{2}+t^{2}\right)^{\frac{1}{4n}}\left(AB+t^{2}+A\sqrt{B^{2}+t^{2}}\right)^{\frac{1}{2}-\frac{1}{4n}}}{\left(B+\sqrt{B^{2}+t^{2}}\right)^{\frac{1}{2}}}\label{eq:9}
\end{equation}
where $A={\displaystyle \frac{x_{1}^{2}}{2}}+{\displaystyle \frac{x_{n+1}^{2}}{2}}+{\displaystyle \frac{1}{2}}\sum_{j=1,j\not=n+1}^{2n}x_{j}^{2}$
and $B={\displaystyle \frac{x_{1}^{2}}{4}+\frac{x_{n+1}^{2}}{4}}+{\displaystyle \frac{1}{2}}\sum_{j=1,j\not=n+1}^{2n}x_{j}^{2}.$\\
Note that for $n=2,$ we get back the fundamental solution as calculated
in \cite{key-3}.

\subsection{Proof of Lemma 2: Bounds for $|\triangledown N|$ and $x\cdot\triangledown N$}

\medskip{}
\hspace{11cm}

Recall that by $|x|={\displaystyle \left(\sum_{j=1}^{2n}x_{j}^{2}\right)^{\frac{1}{2}}}$we
denote the Euclidean norm. In the setting of the anisotropic Heisenberg
group $\mathbb{R}^{2n+1},$ we have the fundamental solution
\[
N=\left(\frac{\left(B^{2}+t^{2}\right)^{\frac{1}{2n}}\left(AB+t^{2}+A\sqrt{B^{2}+t^{2}}\right)^{1-\frac{1}{2n}}}{B+\sqrt{B^{2}+t^{2}}}\right)^{\frac{1}{2}},
\]
where $A={\displaystyle \frac{x_{1}^{2}}{2}}+{\displaystyle \frac{x_{n+1}^{2}}{2}}+{\displaystyle \frac{1}{2}}\sum_{j=1,j\not=n+1}^{2n}x_{j}^{2}$
and $B={\displaystyle \frac{x_{1}^{2}}{4}+\frac{x_{n+1}^{2}}{4}}+{\displaystyle \frac{1}{2}}\sum_{j=1,j\not=n+1}^{2n}x_{j}^{2}.$
In this subsection, we are going to show the following relations:
\begin{equation}
x\cdot\triangledown N\geq-\frac{|x|^{2}}{4nN},\label{eq:14}
\end{equation}
\begin{equation}
|\triangledown N|^{2}\geq\frac{|x|^{2}}{2^{5+\frac{2}{n}}N^{2}},\label{eq:20}
\end{equation}
and
\begin{equation}
|\triangledown N|^{2}\leq\frac{(2n+1)^{2}|x|^{2}}{2^{3}n^{2}N^{2}}.\label{eq:24}
\end{equation}

\begin{proof}
We first calculate $\partial_{x_{j}}N$ and $\partial_{t}N.$

For $j=1$ and $j=n+1:$
\[
\partial_{x_{j}}N=\frac{x_{j}}{4nN}\left(\frac{B\left(B^{2}+t^{2}\right)^{\frac{1}{2n}}\left(AB+t^{2}+A\sqrt{B^{2}+t^{2}}\right)}{\left(B+\sqrt{B^{2}+t^{2}}\right)(B^{2}+t^{2})\left(AB+t^{2}+A\sqrt{B^{2}+t^{2}}\right)^{\frac{1}{2n}}}\right)
\]

\[
+\frac{x_{j}(2n-1)}{4nN}\left(\frac{\left(B^{2}+t^{2}\right)^{\frac{1}{2n}}\left(\frac{A}{2}\sqrt{B^{2}+t^{2}}+B\sqrt{B^{2}+t^{2}}+B^{2}+t^{2}+\frac{AB}{2}\right)}{\left(B+\sqrt{B^{2}+t^{2}}\right)\left(AB+t^{2}+A\sqrt{B^{2}+t^{2}}\right)^{\frac{1}{2n}}(B^{2}+t^{2})^{\frac{1}{2}}}\right)
\]

\[
-\frac{x_{j}}{4N}\left(\frac{\left(B^{2}+t^{2}\right)^{\frac{1}{2n}}\left(AB+t^{2}+A\sqrt{B^{2}+t^{2}}\right)}{(B+\sqrt{B^{2}+t^{2}})(B^{2}+t^{2})^{\frac{1}{2}}\left(AB+t^{2}+A\sqrt{B^{2}+t^{2}}\right)^{\frac{1}{2n}}}\right)
\]
\begin{equation}
=\frac{x_{j}\left(B^{2}+t^{2}\right)^{\frac{1}{2n}}}{4nN}\label{eq:9-1}
\end{equation}
\[
.\left(\frac{\frac{1}{2}B^{2}A+(B-\frac{A}{2})t^{2}+\frac{1}{2}\sqrt{B^{2}+t^{2}}AB+(n-1)(B^{2}+t^{2})\left(B+\sqrt{B^{2}+t^{2}}\right)+nB\sqrt{B^{2}+t^{2}}\left(B+\sqrt{B^{2}+t^{2}}\right)}{\left(B+\sqrt{B^{2}+t^{2}}\right)(B^{2}+t^{2})\left(AB+t^{2}+A\sqrt{B^{2}+t^{2}}\right)^{\frac{1}{2n}}}\right)
\]
Notice that
\begin{equation}
x_{j}\partial_{x_{j}}N\geq0\;\;\;\;\;\;\;\;\;\;\;\;\;\;\;for\;j=1,n+1\label{eq:10}
\end{equation}
To get a bound from above for $|\partial_{x_{j}}N|$, since $B\leq\sqrt{B^{2}+t^{2}}$
, ${\displaystyle B-\frac{A}{2}\leq\frac{A}{2},}$ and ${\displaystyle B^{2}+t^{2}\leq}AB+t^{2}+A\sqrt{B^{2}+t^{2}},$
\[
|\partial_{x_{j}}N|\leq\frac{|x_{j}|}{4nN}\left(\frac{\frac{1}{2}B^{2}A+(\frac{A}{2})t^{2}+\frac{1}{2}\sqrt{B^{2}+t^{2}}AB+(2n-1)(B^{2}+t^{2})\left(B+\sqrt{B^{2}+t^{2}}\right)}{\left(B+\sqrt{B^{2}+t^{2}}\right)(B^{2}+t^{2})}\right)
\]
since $A\leq2B,$ so ${\displaystyle \frac{1}{2}B^{2}A+(\frac{A}{2})t^{2}+\frac{1}{2}\sqrt{B^{2}+t^{2}}AB\leq B\sqrt{B^{2}+t^{2}}\left(B+\sqrt{B^{2}+t^{2}}\right)}\leq(B^{2}+t^{2})\left(B+\sqrt{B^{2}+t^{2}}\right),$
\begin{equation}
|\partial_{x_{j}}N|\leq\frac{|x_{j}|}{2N}\;\;\;\;\;\;\;\;\;\;\;\;\;\;\;for\;j=1,n+1.\label{eq:12}
\end{equation}
Now we calculate $\partial_{x_{j}}N$ for $j=2,..,n,n+2,..,2n:$
\[
\partial_{x_{j}}N=\frac{x_{j}}{4nN}\left(\frac{2B\left(B^{2}+t^{2}\right)^{\frac{1}{2n}}\left(AB+t^{2}+A\sqrt{B^{2}+t^{2}}\right)}{\left(B+\sqrt{B^{2}+t^{2}}\right)(B^{2}+t^{2})\left(AB+t^{2}+A\sqrt{B^{2}+t^{2}}\right)^{\frac{1}{2n}}}\right)
\]

\[
+\frac{x_{j}(2n-1)}{4nN}\left(\frac{\left(B^{2}+t^{2}\right)^{\frac{1}{2n}}\left(A\sqrt{B^{2}+t^{2}}+B\sqrt{B^{2}+t^{2}}+B^{2}+t^{2}+AB\right)}{\left(B+\sqrt{B^{2}+t^{2}}\right)\left(AB+t^{2}+A\sqrt{B^{2}+t^{2}}\right)^{\frac{1}{2n}}(B^{2}+t^{2})^{\frac{1}{2}}}\right)
\]

\[
-\frac{x_{j}}{4N}\left(\frac{2\left(B^{2}+t^{2}\right)^{\frac{1}{2n}}\left(AB+t^{2}+A\sqrt{B^{2}+t^{2}}\right)}{(B+\sqrt{B^{2}+t^{2}})(B^{2}+t^{2})^{\frac{1}{2}}\left(AB+t^{2}+A\sqrt{B^{2}+t^{2}}\right)^{\frac{1}{2n}}}\right)
\]
\begin{equation}
=\frac{x_{j}\left(B^{2}+t^{2}\right)^{\frac{1}{2n}}}{4nN}\left(\frac{AB\sqrt{B^{2}+t^{2}}+AB^{2}+(2B-A)t^{2}+(2n-1)B\sqrt{B^{2}+t^{2}}\left(B+\sqrt{B^{2}+t^{2}}\right)-t^{2}\sqrt{B^{2}+t^{2}}}{\left(B+\sqrt{B^{2}+t^{2}}\right)(B^{2}+t^{2})\left(AB+t^{2}+A\sqrt{B^{2}+t^{2}}\right)^{\frac{1}{2n}}}\right)\label{eq:13-1}
\end{equation}
Notice that
\[
x_{j}\partial_{x_{j}}N\geq\frac{x_{j}^{2}\left(B^{2}+t^{2}\right)^{\frac{1}{2n}}}{4nN}\left(\frac{-t^{2}}{\left(B+\sqrt{B^{2}+t^{2}}\right)\sqrt{B^{2}+t^{2}}\left(AB+t^{2}+A\sqrt{B^{2}+t^{2}}\right)^{\frac{1}{2n}}}\right)
\]
since ${\displaystyle t^{2}\leq}\left(B+\sqrt{B^{2}+t^{2}}\right)\sqrt{B^{2}+t^{2}}$
and ${\displaystyle B^{2}+t^{2}\leq}AB+t^{2}+A\sqrt{B^{2}+t^{2}},$
\begin{equation}
x_{j}\partial_{x_{j}}N\geq-\frac{x_{j}^{2}}{4nN}.\label{eq:13}
\end{equation}
We will now be able to get a lower bound for $x\cdot\triangledown N.$
In fact,
\[
x\cdot\triangledown N=x_{1}\left(\partial_{x_{1}}N-\frac{x_{n+1}}{2}\partial_{t}N\right)+x_{n+1}\left(\partial_{x_{n+1}}N+\frac{x_{1}}{2}\partial_{t}N\right)+\sum_{j=2}^{n}x_{j}\left(\partial_{x_{j}}N-x_{j+n}\partial_{t}N\right)+\sum_{j=n+2}^{2n}x_{j}\left(\partial_{x_{j}}N+x_{j-n}\partial_{t}N\right)
\]
\[
=x_{1}\partial_{x_{1}}N+x_{n+1}\partial_{x_{n+1}}N+\sum_{j=2,j\not=n+1}^{2n}x_{j}\partial_{x_{j}}N
\]
using (\ref{eq:10}) and (\ref{eq:13}),
\[
\geq-\sum_{j=2,j\not=n+1}^{2n}\frac{x_{j}^{2}}{4nN}
\]
\[
\geq-\sum_{j=1}^{2n}\frac{x_{j}^{2}}{4nN}.
\]
Using the above defined Euclidean norm $|x|={\displaystyle \left(\sum_{j=1}^{2n}x_{j}^{2}\right)^{\frac{1}{2}}}$,
\[
x\cdot\triangledown N\geq-\frac{|x|^{2}}{4nN}.
\]
It remains to bound $|\triangledown N|.$
\[
|\triangledown N|^{2}=\left(\partial_{x_{1}}N-\frac{x_{n+1}}{2}\partial_{t}N\right)^{2}+\left(\partial_{x_{n+1}}N+\frac{x_{1}}{2}\partial_{t}N\right)^{2}+\sum_{j=2}^{n}\left(\partial_{x_{j}}N-x_{j+n}\partial_{t}N\right)^{2}+\sum_{j=n+2}^{2n}\left(\partial_{x_{j}}N+x_{j-n}\partial_{t}N\right)^{2}
\]
\begin{equation}
=\sum_{j=1}^{2n}(\partial_{x_{j}}N)^{2}+\frac{x_{1}^{2}}{4}(\partial_{t}N)^{2}+\frac{x_{n+1}^{2}}{4}(\partial_{t}N)^{2}+\sum_{j=2,j\not=n+1}^{2n}x_{j}^{2}(\partial_{t}N)^{2}.\label{eq:15}
\end{equation}
To bound (\ref{eq:15}) from below, we use the fact that for each
$j=2,3,...,n,n+2,..,2n,$
\[
|\partial_{x_{j}}N|^{2}+x_{j}^{2}|\partial_{t}N|^{2}\geq\frac{1}{2}\left(|\partial_{x_{j}}N|+|x_{j}||\partial_{t}N|\right)^{2},
\]
in addition to that we have for $j=1,n+1,$
\[
|\partial_{x_{j}}N|^{2}+\frac{x_{j}^{2}}{4}|\partial_{t}N|^{2}\geq\frac{1}{2}\left(|\partial_{x_{j}}N|+\frac{|x_{j}|}{2}|\partial_{t}N|\right)^{2}
\]
so,
\begin{equation}
|\triangledown N|^{2}\geq\frac{1}{2}\left(|\partial_{x_{1}}N|+\frac{|x_{1}|}{2}|\partial_{t}N|\right)^{2}+\frac{1}{2}\left(|\partial_{x_{n+1}}N|+\frac{|x_{n+1}|}{2}|\partial_{t}N|\right)^{2}+\frac{1}{2}\sum_{j=2,j\not=n+1}^{2n}\left(|\partial_{x_{j}}N|+|x_{j}||\partial_{t}N|\right)^{2}.\label{eq:16}
\end{equation}
Now we proceed to calculate $\partial_{t}N:$
\[
\partial_{t}N=\frac{1}{2N}\left(\frac{t}{2n}\frac{2\left(B^{2}+t^{2}\right)^{\frac{1}{2n}}\left(AB+t^{2}+A\sqrt{B^{2}+t^{2}}\right)}{\left(B+\sqrt{B^{2}+t^{2}}\right)(B^{2}+t^{2})\left(AB+t^{2}+A\sqrt{B^{2}+t^{2}}\right)^{\frac{1}{2n}}}\right)
\]
\[
+\frac{1}{2N}\left(t\left(\frac{2n-1}{2n}\right)\frac{\left(B^{2}+t^{2}\right)^{\frac{1}{2n}}(2\sqrt{B^{2}+t^{2}}+A)}{\left(B+\sqrt{B^{2}+t^{2}}\right)\left(\sqrt{B^{2}+t^{2}}\right)\left(AB+t^{2}+A\sqrt{B^{2}+t^{2}}\right)^{\frac{1}{2n}}}\right)
\]
\[
-\frac{t}{4N}\left(\frac{2\left(B^{2}+t^{2}\right)^{\frac{1}{2n}}\left(AB+t^{2}+A\sqrt{B^{2}+t^{2}}\right)}{\left(B+\sqrt{B^{2}+t^{2}}\right)^{2}\left(\sqrt{B^{2}+t^{2}}\right)\left(AB+t^{2}+A\sqrt{B^{2}+t^{2}}\right)^{\frac{1}{2n}}}\right)
\]
\[
=\frac{t\left(B^{2}+t^{2}\right)^{\frac{1}{2n}}}{4nN}\left(\frac{2\left(AB+t^{2}+A\sqrt{B^{2}+t^{2}}\right)+(2n-1)\sqrt{B^{2}+t^{2}}\left(2\sqrt{B^{2}+t^{2}}+A\right)-2n\frac{\sqrt{B^{2}+t^{2}}\left(AB+t^{2}+A\sqrt{B^{2}+t^{2}}\right)}{\left(B+\sqrt{B^{2}+t^{2}}\right)}}{\left(B+\sqrt{B^{2}+t^{2}}\right)(B^{2}+t^{2})\left(AB+t^{2}+A\sqrt{B^{2}+t^{2}}\right)^{\frac{1}{2n}}}\right)
\]
\begin{equation}
=\frac{t\left(B^{2}+t^{2}\right)^{\frac{1}{2n}}}{4nN}\left(\frac{2B(A-B)+A\sqrt{B^{2}+t^{2}}+2n\frac{2B^{3}+2t^{2}B+2B^{2}\sqrt{B^{2}+t^{2}}+t^{2}\sqrt{B^{2}+t^{2}}}{\left(B+\sqrt{B^{2}+t^{2}}\right)}}{\left(B+\sqrt{B^{2}+t^{2}}\right)(B^{2}+t^{2})\left(AB+t^{2}+A\sqrt{B^{2}+t^{2}}\right)^{\frac{1}{2n}}}\right).\label{eq:17}
\end{equation}
For $j=2,..,n,n+2,..,2n,$ we use the reverse triangle inequality
on (\ref{eq:13-1}) in addition to (\ref{eq:17}),
\[
|\partial_{x_{j}}N|+|x_{j}||\partial_{t}N|\geq\frac{|x_{j}|\left(B^{2}+t^{2}\right)^{\frac{1}{2n}}}{4nN}
\]
\[
.\left(\frac{AB\sqrt{B^{2}+t^{2}}+AB^{2}+(2B-A)t^{2}+(2n-1)B\sqrt{B^{2}+t^{2}}\left(B+\sqrt{B^{2}+t^{2}}\right)-t^{2}\sqrt{B^{2}+t^{2}}}{\left(B+\sqrt{B^{2}+t^{2}}\right)(B^{2}+t^{2})\left(AB+t^{2}+A\sqrt{B^{2}+t^{2}}\right)^{\frac{1}{2n}}}\right)
\]
\[
+\frac{|x_{j}|\left(B^{2}+t^{2}\right)^{\frac{1}{2n}}}{4nN}\left(\frac{2B(A-B)|t|+A|t|\sqrt{B^{2}+t^{2}}+2n|t|\frac{2B^{3}+2t^{2}B+2B^{2}\sqrt{B^{2}+t^{2}}+t^{2}\sqrt{B^{2}+t^{2}}}{\left(B+\sqrt{B^{2}+t^{2}}\right)}}{\left(B+\sqrt{B^{2}+t^{2}}\right)(B^{2}+t^{2})\left(AB+t^{2}+A\sqrt{B^{2}+t^{2}}\right)^{\frac{1}{2n}}}\right)
\]

\[
=\frac{|x_{j}|\left(B^{2}+t^{2}\right)^{\frac{1}{2n}}}{4nN}
\]
\[
.\left(\frac{AB\sqrt{B^{2}+t^{2}}+A|t|\sqrt{B^{2}+t^{2}}+AB^{2}+2AB|t|-At^{2}-B^{2}\sqrt{B^{2}+t^{2}}-t^{2}\sqrt{B^{2}+t^{2}}-2B^{2}|t|-B^{3}+Bt^{2}}{\left(B+\sqrt{B^{2}+t^{2}}\right)(B^{2}+t^{2})\left(AB+t^{2}+A\sqrt{B^{2}+t^{2}}\right)^{\frac{1}{2n}}}\right)
\]
\[
+\frac{|x_{j}|\left(B^{2}+t^{2}\right)^{\frac{1}{2n}}}{4nN}\left(\frac{n\left(2B^{2}\sqrt{B^{2}+t^{2}}+2B|t|\sqrt{B^{2}+t^{2}}+2B^{2}|t|+2B^{3}+2Bt^{2}+2|t|^{3}\right)}{\left(B+\sqrt{B^{2}+t^{2}}\right)(B^{2}+t^{2})\left(AB+t^{2}+A\sqrt{B^{2}+t^{2}}\right)^{\frac{1}{2n}}}\right)
\]

since $A|t|\sqrt{B^{2}+t^{2}}\geq At^{2}$, $AB\sqrt{B^{2}+t^{2}}\geq B^{2}\sqrt{B^{2}+t^{2}},$
$AB^{2}\geq B^{3},$ and $2AB|t|\geq2B^{2}|t|,$

\[
\geq\frac{|x_{j}|\left(B^{2}+t^{2}\right)^{\frac{1}{2n}}}{4nN}\left(\frac{-t^{2}\sqrt{B^{2}+t^{2}}+Bt^{2}+2n\left(B^{2}\sqrt{B^{2}+t^{2}}+B|t|\sqrt{B^{2}+t^{2}}+B^{2}t+B^{3}+t^{2}(B+|t|)\right)}{\left(B+\sqrt{B^{2}+t^{2}}\right)(B^{2}+t^{2})\left(AB+t^{2}+A\sqrt{B^{2}+t^{2}}\right)^{\frac{1}{2n}}}\right)
\]
since $B+|t|\geq\sqrt{B^{2}+t^{2}},$
\[
\geq\frac{|x_{j}|(2n-1)\left(B^{2}+t^{2}\right)^{\frac{1}{2n}}}{4nN}\left(\frac{B\sqrt{B^{2}+t^{2}}(B+|t|)+B^{2}(|t|+B)+t^{2}\sqrt{B^{2}+t^{2}}}{\left(B+\sqrt{B^{2}+t^{2}}\right)(B^{2}+t^{2})\left(AB+t^{2}+A\sqrt{B^{2}+t^{2}}\right)^{\frac{1}{2n}}}\right)
\]
since $B+|t|\geq\sqrt{B^{2}+t^{2}},$
\[
\geq\frac{|x_{j}|(2n-1)\left(B^{2}+t^{2}\right)^{\frac{1}{2n}}}{4nN}\left(\frac{B^{2}+t^{2}}{AB+t^{2}+A\sqrt{B^{2}+t^{2}}}\right)^{\frac{1}{2n}}.
\]
Since $B^{2}+t^{2}\geq{\displaystyle \frac{1}{4}}\left(AB+t^{2}+A\sqrt{B^{2}+t^{2}}\right),$
then,
\begin{equation}
|\partial_{x_{j}}N|+|x_{j}||\partial_{t}N|\geq(2n-1)\frac{|x_{j}|}{2^{\frac{1}{n}+2}nN}.\label{eq:18}
\end{equation}
For $j=1,n+1,$ using the calculations (\ref{eq:9-1}) and (\ref{eq:17}),
\[
|\partial_{x_{j}}N|+\frac{|x_{j}|}{2}|\partial_{t}N|
\]
\[
=\frac{|x_{j}|\left(B^{2}+t^{2}\right)^{\frac{1}{2n}}}{4nN}
\]
\[
.\left(\frac{\frac{1}{2}B^{2}A+(B-\frac{A}{2})t^{2}+\frac{1}{2}\sqrt{B^{2}+t^{2}}AB+(n-1)(B^{2}+t^{2})\left(B+\sqrt{B^{2}+t^{2}}\right)+nB\sqrt{B^{2}+t^{2}}\left(B+\sqrt{B^{2}+t^{2}}\right)}{\left(B+\sqrt{B^{2}+t^{2}}\right)(B^{2}+t^{2})\left(AB+t^{2}+A\sqrt{B^{2}+t^{2}}\right)^{\frac{1}{2n}}}\right)
\]
\[
+\frac{|x_{j}|\left(B^{2}+t^{2}\right)^{\frac{1}{2n}}}{4nN}\left(\frac{B(A-B)|t|+\frac{1}{2}A|t|\sqrt{B^{2}+t^{2}}+n|t|\frac{2B^{3}+2t^{2}B+2B^{2}\sqrt{B^{2}+t^{2}}+t^{2}\sqrt{B^{2}+t^{2}}}{\left(B+\sqrt{B^{2}+t^{2}}\right)}}{\left(B+\sqrt{B^{2}+t^{2}}\right)(B^{2}+t^{2})\left(AB+t^{2}+A\sqrt{B^{2}+t^{2}}\right)^{\frac{1}{2n}}}\right)
\]

\[
=\frac{|x_{j}|\left(B^{2}+t^{2}\right)^{\frac{1}{2n}}}{4nN}
\]
\[
.\left(\frac{\frac{1}{2}B^{2}A+(B-\frac{A}{2})t^{2}+\frac{1}{2}\sqrt{B^{2}+t^{2}}AB+(n-1)(B^{2}+t^{2})\left(B+\sqrt{B^{2}+t^{2}}\right)+nB\sqrt{B^{2}+t^{2}}\left(B+\sqrt{B^{2}+t^{2}}\right)}{\left(B+\sqrt{B^{2}+t^{2}}\right)(B^{2}+t^{2})\left(AB+t^{2}+A\sqrt{B^{2}+t^{2}}\right)^{\frac{1}{2n}}}\right)
\]
\[
+\frac{|x_{j}|\left(B^{2}+t^{2}\right)^{\frac{1}{2n}}}{4nN}\left(\frac{B(A-B)|t|+\frac{1}{2}A|t|\sqrt{B^{2}+t^{2}}+n|t|\frac{2B^{3}+2t^{2}B+2B^{2}\sqrt{B^{2}+t^{2}}+t^{2}\sqrt{B^{2}+t^{2}}}{\left(B+\sqrt{B^{2}+t^{2}}\right)}}{\left(B+\sqrt{B^{2}+t^{2}}\right)(B^{2}+t^{2})\left(AB+t^{2}+A\sqrt{B^{2}+t^{2}}\right)^{\frac{1}{2n}}}\right)
\]

\[
=\frac{|x_{j}|\left(B^{2}+t^{2}\right)^{\frac{1}{2n}}}{4nN}
\]
\[
.\left(\frac{\frac{1}{2}AB\sqrt{B^{2}+t^{2}}+\frac{1}{2}A|t|\sqrt{B^{2}+t^{2}}+\frac{AB^{2}}{2}+AB|t|-\frac{At^{2}}{2}-B^{2}\sqrt{B^{2}+t^{2}}-t^{2}\sqrt{B^{2}+t^{2}}-B^{2}|t|-B^{3}}{\left(B+\sqrt{B^{2}+t^{2}}\right)(B^{2}+t^{2})\left(AB+t^{2}+A\sqrt{B^{2}+t^{2}}\right)^{\frac{1}{2n}}}\right)
\]
\[
+\frac{|x_{j}|\left(B^{2}+t^{2}\right)^{\frac{1}{2n}}}{4nN}\left(\frac{n\left(2B^{2}\sqrt{B^{2}+t^{2}}+B|t|\sqrt{B^{2}+t^{2}}+t^{2}\sqrt{B^{2}+t^{2}}+B^{2}|t|+2B^{3}+2Bt^{2}+|t|^{3}\right)}{\left(B+\sqrt{B^{2}+t^{2}}\right)(B^{2}+t^{2})\left(AB+t^{2}+A\sqrt{B^{2}+t^{2}}\right)^{\frac{1}{2n}}}\right)
\]

since $\frac{1}{2}A|t|\sqrt{B^{2}+t^{2}}\geq\frac{At^{2}}{2},$ $AB|t|\geq B^{2}|t|,$
and $\frac{1}{2}AB\sqrt{B^{2}+t^{2}}+\frac{AB^{2}}{2}\geq B^{3},$
\[
\geq\frac{|x_{j}|\left(B^{2}+t^{2}\right)^{\frac{1}{2n}}}{4nN}\left(\frac{-B^{2}\sqrt{B^{2}+t^{2}}-t^{2}\sqrt{B^{2}+t^{2}}}{\left(B+\sqrt{B^{2}+t^{2}}\right)(B^{2}+t^{2})\left(AB+t^{2}+A\sqrt{B^{2}+t^{2}}\right)^{\frac{1}{2n}}}\right)
\]
\[
+\frac{|x_{j}|\left(B^{2}+t^{2}\right)^{\frac{1}{2n}}}{4nN}\left(\frac{n\left(2B^{2}\sqrt{B^{2}+t^{2}}+B|t|\sqrt{B^{2}+t^{2}}+t^{2}\sqrt{B^{2}+t^{2}}+B^{2}|t|+2B^{3}+2Bt^{2}+|t|^{3}\right)}{\left(B+\sqrt{B^{2}+t^{2}}\right)(B^{2}+t^{2})\left(AB+t^{2}+A\sqrt{B^{2}+t^{2}}\right)^{\frac{1}{2n}}}\right)
\]
\[
=\frac{|x_{j}|\left(B^{2}+t^{2}\right)^{\frac{1}{2n}}}{4nN}\left(\frac{(2n-1)B^{2}\sqrt{B^{2}+t^{2}}+(n-1)t^{2}\sqrt{B^{2}+t^{2}}+n\left(B|t|\sqrt{B^{2}+t^{2}}+(B^{2}+t^{2})(B+|t|+B)\right)}{\left(B+\sqrt{B^{2}+t^{2}}\right)(B^{2}+t^{2})\left(AB+t^{2}+A\sqrt{B^{2}+t^{2}}\right)^{\frac{1}{2n}}}\right)
\]
\[
\geq\frac{|x_{j}|\left(B^{2}+t^{2}\right)^{\frac{1}{2n}}}{4N}\left(\frac{(B^{2}+t^{2})(B+|t|+B)}{\left(B+\sqrt{B^{2}+t^{2}}\right)(B^{2}+t^{2})\left(AB+t^{2}+A\sqrt{B^{2}+t^{2}}\right)^{\frac{1}{2n}}}\right)
\]
since $|t|+B\geq\sqrt{B^{2}+t^{2}},$
\[
\geq\frac{|x_{j}|}{4N}\left(\frac{B^{2}+t^{2}}{AB+t^{2}+A\sqrt{B^{2}+t^{2}}}\right)^{\frac{1}{2n}}
\]
Since $B^{2}+t^{2}\geq{\displaystyle \frac{1}{4}}\left(AB+t^{2}+A\sqrt{B^{2}+t^{2}}\right),$
then for $j=1,n+1$,
\begin{equation}
|\partial_{x_{j}}N|+\frac{|x_{j}|}{2}|\partial_{t}N|\geq\frac{|x_{j}|}{2^{2+\frac{1}{n}}N}.\label{eq:19}
\end{equation}
Replacing (\ref{eq:18}) and (\ref{eq:19}) in (\ref{eq:16}) we get
a lower bound for $|\triangledown N|^{2}:$
\[
|\triangledown N|^{2}\geq\frac{1}{2}\left(|\partial_{x_{1}}N|+\frac{|x_{1}|}{2}|\partial_{t}N|\right)^{2}+\frac{1}{2}\left(|\partial_{x_{n+1}}N|+\frac{|x_{n+1}|}{2}|\partial_{t}N|\right)^{2}+\frac{1}{2}\sum_{j=2,j\not=n+1}^{2n}\left(|\partial_{x_{j}}N|+|x_{j}||\partial_{t}N|\right)^{2}
\]
\[
\geq\frac{|x_{1}|^{2}}{2^{5+\frac{2}{n}}N^{2}}+\frac{|x_{n+1}|^{2}}{2^{5+\frac{2}{n}}N^{2}}+\frac{(2n-1)^{2}}{2^{5+\frac{2}{n}}n^{2}N^{2}}\sum_{j=2,j\not=n+1}^{2n}|x_{j}|^{2}.
\]
Hence, using the defined Euclidean norm $|x|={\displaystyle \left(\sum_{j=1}^{2n}x_{j}^{2}\right)^{\frac{1}{2}}},$
\[
|\triangledown N|^{2}\geq\frac{|x|^{2}}{2^{5+\frac{2}{n}}N^{2}}.
\]
The last thing to do in this subsection is to obtain an upper bound
for $|\triangledown N|^{2}.$ By (\ref{eq:15}),
\[
|\triangledown N|^{2}=\sum_{j=1}^{2n}(\partial_{x_{j}}N)^{2}+\frac{x_{1}^{2}}{4}(\partial_{t}N)^{2}+\frac{x_{n+1}^{2}}{4}(\partial_{t}N)^{2}+\sum_{j=2,j\not=n+1}^{2n}x_{j}^{2}(\partial_{t}N)^{2}
\]
\begin{equation}
\leq\sum_{j=1}^{2n}(\partial_{x_{j}}N)^{2}+|x|^{2}(\partial_{t}N)^{2}.\label{eq:21}
\end{equation}
For $j=2,..,n,n+2,...,2n,$ from (\ref{eq:13-1}) and the triangle
inequality,
\[
|\partial_{x_{j}}N|=\frac{|x_{j}|\left(B^{2}+t^{2}\right)^{\frac{1}{2n}}}{4nN}
\]
\[
.\left(\frac{AB\sqrt{B^{2}+t^{2}}+AB^{2}+(2B-A)t^{2}+(2n-1)B\sqrt{B^{2}+t^{2}}\left(B+\sqrt{B^{2}+t^{2}}\right)+t^{2}\sqrt{B^{2}+t^{2}}}{\left(B+\sqrt{B^{2}+t^{2}}\right)(B^{2}+t^{2})\left(AB+t^{2}+A\sqrt{B^{2}+t^{2}}\right)^{\frac{1}{2n}}}\right)
\]
since $\left(B^{2}+t^{2}\right)\leq\left(AB+t^{2}+A\sqrt{B^{2}+t^{2}}\right),$
\[
\leq\frac{|x_{j}|}{4nN}\left(\frac{(2n-1)B^{3}+(2n-1)B^{2}\sqrt{B^{2}+t^{2}}+(2n+1)Bt^{2}+AB\sqrt{B^{2}+t^{2}}+AB^{2}-At^{2}+t^{2}\sqrt{B^{2}+t^{2}}}{B^{3}+B^{2}\sqrt{B^{2}+t^{2}}+Bt^{2}+t^{2}\sqrt{B^{2}+t^{2}}}\right)
\]
since $B\leq A\leq2B,$
\[
\leq\frac{|x_{j}|}{4nN}\left(\frac{(2n+1)B^{3}+(2n+1)B^{2}\sqrt{B^{2}+t^{2}}+(2n)Bt^{2}+t^{2}\sqrt{B^{2}+t^{2}}}{B^{3}+B^{2}\sqrt{B^{2}+t^{2}}+Bt^{2}+t^{2}\sqrt{B^{2}+t^{2}}}\right).
\]
So, 
\begin{equation}
|\partial_{x_{j}}N|\leq\frac{(2n+1)|x_{j}|}{4nN}.\label{eq:22}
\end{equation}
From (\ref{eq:17}),
\[
|\partial_{t}N|=\frac{|t|\left(B^{2}+t^{2}\right)^{\frac{1}{2n}}}{4nN}\left(\frac{2B(A-B)+A\sqrt{B^{2}+t^{2}}+2n\frac{2B^{3}+2t^{2}B+2B^{2}\sqrt{B^{2}+t^{2}}+t^{2}\sqrt{B^{2}+t^{2}}}{\left(B+\sqrt{B^{2}+t^{2}}\right)}}{\left(B+\sqrt{B^{2}+t^{2}}\right)(B^{2}+t^{2})\left(AB+t^{2}+A\sqrt{B^{2}+t^{2}}\right)^{\frac{1}{2n}}}\right)
\]
since $\left(B^{2}+t^{2}\right)\leq\left(AB+t^{2}+A\sqrt{B^{2}+t^{2}}\right),$
\[
\leq\frac{|t|}{4nN}\left(\frac{2B(A-B)+A\sqrt{B^{2}+t^{2}}+2n\frac{2B^{3}+2t^{2}B+2B^{2}\sqrt{B^{2}+t^{2}}+t^{2}\sqrt{B^{2}+t^{2}}}{\left(B+\sqrt{B^{2}+t^{2}}\right)}}{\left(B+\sqrt{B^{2}+t^{2}}\right)(B^{2}+t^{2})}\right)
\]
since $|t|\leq\sqrt{B^{2}+t^{2}},$
\[
\leq\frac{1}{4nN}\left(\frac{\left(B+\sqrt{B^{2}+t^{2}}\right)\left(2B(A-B)+A\sqrt{B^{2}+t^{2}}\right)+2n(2B^{3}+2t^{2}B+2B^{2}\sqrt{B^{2}+t^{2}}+t^{2}\sqrt{B^{2}+t^{2}})}{\left(B+\sqrt{B^{2}+t^{2}}\right)^{2}\sqrt{B^{2}+t^{2}}}\right)
\]
\[
=\frac{1}{4nN}\left(\frac{3B^{2}A-2B^{3}+3AB\sqrt{B^{2}+t^{2}}-2B^{2}\sqrt{B^{2}+t^{2}}+At^{2}+2n\left(2B^{3}+2t^{2}B+2B^{2}\sqrt{B^{2}+t^{2}}+t^{2}\sqrt{B^{2}+t^{2}}\right)}{2B^{3}+2t^{2}B+2B^{2}\sqrt{B^{2}+t^{2}}+t^{2}\sqrt{B^{2}+t^{2}}}\right)
\]
since $A\leq{\displaystyle \frac{B}{2}}$
\[
\leq\frac{1}{4nN}\left(\frac{-\frac{1}{2}B^{3}-\frac{1}{2}B^{2}\sqrt{B^{2}+t^{2}}+At^{2}+2n\left(2B^{3}+2t^{2}B+2B^{2}\sqrt{B^{2}+t^{2}}+t^{2}\sqrt{B^{2}+t^{2}}\right)}{2B^{3}+2t^{2}B+2B^{2}\sqrt{B^{2}+t^{2}}+t^{2}\sqrt{B^{2}+t^{2}}}\right).
\]
Hence,
\begin{equation}
|\partial_{t}N|\leq\frac{(2n+1)}{4nN}.\label{eq:23}
\end{equation}
Inserting (\ref{eq:12}), (\ref{eq:22}), and (\ref{eq:23}) in (\ref{eq:21}),
we get: 
\[
|\triangledown N|^{2}\leq\sum_{j=1}^{2n}(\partial_{x_{j}}N)^{2}+|x|^{2}(\partial_{t}N)^{2}
\]
\[
\leq\frac{|x_{1}|^{2}}{4N^{2}}+\frac{|x_{n+1}|^{2}}{4N^{2}}+\sum_{j=2,j\not=n+1}^{2n}\frac{(2n+1)^{2}|x_{j}|^{2}}{2^{4}n^{2}N^{2}}+|x|^{2}\frac{(2n+1)^{2}}{2^{4}n^{2}N^{2}}.
\]
So,
\[
|\triangledown N|^{2}\leq\frac{(2n+1)^{2}|x|^{2}}{2^{3}n^{2}N^{2}}.
\]
\end{proof}

\section{U-Bound }

The following U-Bound will be used in section 4 to prove the $q-$Poincaré
inequality for $q\ensuremath{\geq2}$ for the measure
\[
d\mu=\frac{e^{-g\left(N\right)}}{Z}d\lambda
\]
under the condition that ${\displaystyle \frac{g'\left(N\right)}{N^{2}}}$
is an increasing function on the anisotropic Heisenberg group where
$n>5.$ In addition, in section 4, we will prove a $\beta-$Logarithmic
Sobolev inequality for ${\displaystyle d\mu=\frac{e^{-\alpha N^{p}}}{Z}d\lambda,}$
for $p\geq4,$ $q\geq2,$ and $0<\beta\leq\frac{p-3}{p}.$
\begin{thm}
Let $N^{-2n}$ be the fundamental solution in the setting of the anisotropic
Heisenberg group $\mathbb{R}^{2n+1}$ with $n>5.$ Let $g:\left[0,\infty\right)\rightarrow\left[0,\infty\right)$
be a differentiable increasing function such that $g''(N)\leq g'(N)^{2}$
on $\{N\geq1\}.$ Let ${\displaystyle d\mu=\frac{e^{-g\left(N\right)}}{Z}d\lambda}$
be a probability measure and $Z$ the normalization constant. Then,
for $q\ensuremath{\geq2,}$ 
\[
\int\frac{g'\left(N\right)}{N^{2}}\vert f\vert^{q}d\mu\leq C\int\vert\triangledown f\vert^{q}d\mu+D\int\vert f\vert^{q}d\mu
\]
holds outside the unit ball $\left\{ N<1\right\} $ with $C$ and
$D$ positive constants independent of a function $f$ for which the
right hand side is well defined. 
\end{thm}

\begin{proof}
First, we prove the result for $q=2:$

Using integration by parts,

\[
\int\left(\triangledown N\right)\cdot\left(\triangledown f\right)e^{-g\left(N\right)}d\lambda=-\int\triangledown\left(\triangledown Ne^{-g\left(N\right)}\right)fd\lambda=-\int\Delta Nfe^{-g\left(N\right)}d\lambda+\int\vert\triangledown N\vert^{2}fg^{'}\left(N\right)e^{-g\left(N\right)}d\lambda.
\]
Since $N^{-2n}$ is the fundamental solution, we have
\[
\Delta N=\vert\triangledown N\vert^{2}\frac{\left(Q-1\right)}{N}.
\]
Hence,
\begin{equation}
\int\vert\triangledown N\vert^{2}f\left[g^{'}\left(N\right)-\frac{\left(Q-1\right)}{N}\right]e^{-g\left(N\right)}d\lambda=\int\left(\triangledown N\right)\cdot\left(\triangledown f\right)e^{-g\left(N\right)}d\lambda,\label{eq:15-1}
\end{equation}
Replacing $f$ by ${\displaystyle \frac{f^{2}}{\vert x\vert^{2}}}$
and using (\ref{eq:20-2}), the left-hand side of (\ref{eq:15-1})
becomes:
\[
\int\vert\triangledown N\vert^{2}\frac{f^{2}}{\vert x\vert^{2}}\left[g'\left(N\right)-\frac{\left(Q-1\right)}{N}\right]e^{-g\left(N\right)}d\lambda\geq\frac{1}{2^{5+\frac{2}{n}}}\int f^{2}\left[\frac{g'\left(N\right)}{N^{2}}-\frac{\left(Q-1\right)}{N^{3}}\right]e^{-g\left(N\right)}d\lambda
\]
\begin{equation}
\geq\frac{1}{2^{5+\frac{2}{n}}}\int f^{2}\left[\frac{g^{'}\left(N\right)}{N^{2}}-\left(Q-1\right)\right]e^{-g\left(N\right)}d\lambda.\label{eq:23-1}
\end{equation}
Where the last inequality is true since $N>1.$ As for the right-hand
side of (\ref{eq:15-1}),
\[
\int\left(\triangledown N\right)\cdot\left(\triangledown\left(\frac{f^{2}}{\vert x\vert^{2}}\right)\right)e^{-g\left(N\right)}d\lambda=\int\left(\triangledown N\right)\cdot\left[2f\frac{\triangledown f}{\vert x\vert^{2}}-\frac{2f^{2}\triangledown\vert x\vert}{\vert x\vert^{3}}\right]e^{-g\left(N\right)}d\lambda
\]

\begin{equation}
=\int\left(\triangledown N\right).\left[2f\frac{\triangledown f}{\vert x\vert^{2}}-\frac{2f^{2}x}{\vert x\vert^{4}}\right]e^{-g\left(N\right)}d\lambda.\label{eq:16-1}
\end{equation}
Using the bound on $\triangledown N\cdot x,$ from (\ref{eq:14-1}),
we get:
\[
\int\left(\triangledown N\right)\cdot\left(\triangledown\left(\frac{f^{2}}{\vert x\vert^{2}}\right)\right)e^{-g\left(N\right)}d\lambda=\int\frac{2f}{\vert x\vert^{2}}\triangledown N\cdot\triangledown fe^{-g\left(N\right)}d\lambda-2\int f^{2}\frac{\triangledown N\cdot x}{\vert x\vert^{4}}e^{-g\left(N\right)}d\lambda
\]
\[
\leq\int\frac{2f}{\vert x\vert^{2}}\triangledown N\cdot\triangledown fe^{-g\left(N\right)}d\lambda+\frac{1}{2n}\int\frac{f^{2}}{N\vert x\vert^{2}}e^{-g\left(N\right)}d\lambda
\]
\[
\leq2\int\frac{f}{\vert x\vert^{2}}|\triangledown N||\triangledown f|e^{-g\left(N\right)}d\lambda+\frac{1}{2n}\int\frac{f^{2}}{N\vert x\vert^{2}}e^{-g\left(N\right)}d\lambda
\]
using the bound on $|\triangledown N|,$ from (\ref{eq:24-2}), we
get
\[
\leq\frac{(2n+1)}{2^{\frac{1}{2}}n}\int\frac{|f|}{N\vert x\vert}|\triangledown f|e^{-g\left(N\right)}d\lambda+\frac{1}{2n}\int\frac{f^{2}}{N\vert x\vert^{2}}e^{-g\left(N\right)}d\lambda
\]
applying Cauchy\textquoteright s inequality with $\beta:$ ${\displaystyle ab\leq\beta a^{2}+\frac{b^{2}}{4\beta}}$
with ${\displaystyle a=\frac{\vert f\vert}{N\vert x\vert}e^{-\frac{g\left(N\right)}{2}}}$
and ${\displaystyle b=\frac{(2n+1)}{2^{\frac{1}{2}}n}\vert\triangledown f\vert e^{-\frac{g\left(N\right)}{2}},}$
\[
\leq\beta\int\frac{f^{2}}{N^{2}\vert x\vert^{2}}e^{-g\left(N\right)}d\lambda+\frac{1}{2n}\int\frac{f^{2}}{N\vert x\vert^{2}}e^{-g\left(N\right)}d\lambda+\frac{(2n+1)^{2}}{8n^{2}\beta}\int|\triangledown f|^{2}e^{-g\left(N\right)}d\lambda.
\]
Combining the last inequality with (\ref{eq:23-1}), we get 
\begin{equation}
\frac{1}{2^{5+\frac{2}{n}}}\int f^{2}\left[\frac{g^{'}\left(N\right)}{N^{2}}-\left(Q-1\right)\right]e^{-g\left(N\right)}d\lambda\leq\beta\int\frac{f^{2}}{N^{2}\vert x\vert^{2}}e^{-g\left(N\right)}d\lambda+\frac{1}{2n}\int\frac{f^{2}}{N\vert x\vert^{2}}e^{-g\left(N\right)}d\lambda\label{eq:24-1}
\end{equation}
\[
+\frac{(2n+1)^{2}}{8n^{2}\beta}\int|\triangledown f|^{2}e^{-g\left(N\right)}d\lambda.
\]
For $n$ large enough, to be determined later, we need to bound $\frac{1}{2n}\int\frac{f^{2}}{N\vert x\vert^{2}}e^{-g\left(N\right)}d\lambda$
from the right hand side of (\ref{eq:24-1}) by $\frac{1}{2^{5+\frac{2}{n}}}\int f^{2}\frac{g^{'}\left(N\right)}{N^{2}}e^{-g\left(N\right)}d\lambda$
by using Hardy's inequality (see \cite{key-3-1-1} and references
therein) and the Coarea formula (page 468 of \cite{key-11-1}). Since
$\beta$ could be chosen to be arbitrarily small and since $N>1,$
we do not worry about the term $\beta\int\frac{f^{2}}{N^{2}\vert x\vert^{2}}e^{-g\left(N\right)}d\lambda$
since it follows the same procedure as $\frac{1}{2n}\int\frac{f^{2}}{N\vert x\vert^{2}}e^{-g\left(N\right)}d\lambda.$
Let $E=\{(x,z):{\displaystyle \frac{1}{|x|^{2}}\leq\alpha\frac{g'(N)}{N}}\}$
and $F=\{(x,z):{\displaystyle \alpha\frac{|x|^{2}g'(N)}{N}<1}\}.$
\[
\frac{1}{2n}\int\frac{f^{2}}{N\vert x\vert^{2}}e^{-g\left(N\right)}d\lambda=\frac{1}{2n}\int_{F}\frac{f^{2}}{N\vert x\vert^{2}}e^{-g\left(N\right)}d\lambda+\frac{1}{2n}\int_{E}\frac{f^{2}}{N\vert x\vert^{2}}e^{-g\left(N\right)}d\lambda
\]
\begin{equation}
\leq\frac{1}{2n}\int_{F}\frac{f^{2}}{N\vert x\vert^{2}}e^{-g\left(N\right)}d\lambda+\frac{\alpha}{2n}\int_{E}\frac{f^{2}g'(N)}{N^{2}}e^{-g\left(N\right)}d\lambda.\label{eq:19-1}
\end{equation}
where (\ref{eq:19-1}) is true since $E=\{(x,z):\frac{1}{|x|^{2}}\leq\alpha\frac{g'(N)}{N}\}.$
$\alpha$ is to be chosen later. The aim now is to estimate the first
term of (\ref{eq:19-1}). Consider $F_{r}=\left\{ \alpha\frac{|x|^{2}g'(N)}{N}<r\right\} ,$
where $1<r<2.$ Integrating by parts:
\[
\frac{1}{2n}\int_{F_{r}}\frac{\vert fe^{\frac{-g(N)}{2}}\vert^{2}}{N|x|^{2}}d\lambda=\left(\frac{1}{2n}\right)\frac{1}{2n-2}\int_{F_{r}}\frac{\vert fe^{\frac{-g(N)}{2}}\vert^{2}}{N^{2}}\nabla\left(\frac{x}{|x|^{2}}\right)d\lambda
\]
\[
=-\left(\frac{1}{2n}\right)\frac{1}{2n-2}\int_{F_{r}}\nabla\left(\left\vert \frac{fe^{\frac{-g(N)}{2}}}{N^{\frac{1}{2}}}\right\vert ^{2}\right)\cdot\frac{x}{|x|^{2}}d\lambda
\]
\[
+\left(\frac{1}{2n}\right)\frac{1}{2n-2}\int_{\partial F_{r}}\frac{f^{2}e^{-g(N)}}{N|x|^{2}}\sum_{j=1}^{2n}\frac{x_{j}<X_{j}I,\nabla_{euc}\left(\alpha\frac{|x|^{2}g'(N)}{N}\right)>}{\left|\nabla_{euc}\left(\alpha\frac{|x|^{2}g'(N)}{N}\right)\right|}dH^{2n}
\]
\[
=-\left(\frac{1}{2n}\right)\frac{1}{n-1}\int_{F_{r}}\frac{fe^{\frac{-g(N)}{2}}}{N^{\frac{1}{2}}}\nabla\left(\frac{fe^{\frac{-g(N)}{2}}}{N^{\frac{1}{2}}}\right)\cdot\frac{x}{|x|^{2}}d\lambda
\]
\[
+\left(\frac{1}{2n}\right)\frac{1}{2n-2}\int_{\partial F_{r}}\frac{f^{2}e^{-g(N)}}{N|x|^{2}}\sum_{j=1}^{2n}\frac{x_{j}<X_{j}I,\nabla_{euc}\left(\alpha\frac{|x|^{2}g'(N)}{N}\right)>}{\left|\nabla_{euc}\left(\alpha\frac{|x|^{2}g'(N)}{N}\right)\right|}dH^{2n}
\]
\[
\leq\frac{1}{4n}\int_{F_{r}}\frac{\vert fe^{\frac{-g(N)}{2}}\vert^{2}}{N|x|^{2}}d\lambda+\frac{1}{4n(n-1)^{2}}\int_{F_{r}}\left|\nabla\left(\frac{fe^{\frac{-g(N)}{2}}}{N^{\frac{1}{2}}}\right)\right|^{2}d\lambda
\]
\[
+\frac{1}{4n(n-1)}\int_{\partial F_{r}}\frac{f^{2}e^{-g(N)}}{N|x|^{2}}\sum_{j=1}^{2n}\frac{x_{j}<X_{j}I,\nabla_{euc}\left(\alpha\frac{|x|^{2}g'(N)}{N}\right)>}{\left|\nabla_{euc}\left(\alpha\frac{|x|^{2}g'(N)}{N}\right)\right|}dH^{2n}
\]
Where the last step uses Cauchy's inequality. Subtracting on both
sides of the last inequality by ${\displaystyle \frac{1}{4n}\int_{F_{r}}\frac{\vert fe^{\frac{-g(N)}{2}}\vert^{2}}{N|x|^{2}}d\lambda,}$
and using the fact that $1<r<2,$ we get:
\[
\left(\frac{1}{2n}\right)\int_{F_{1}}\frac{\vert fe^{\frac{-g(N)}{2}}\vert^{2}}{N|x|^{2}}d\lambda\leq\left(\frac{1}{2n}\right)\int_{F_{r}}\frac{\vert fe^{\frac{-g(N)}{2}}\vert^{2}}{N|x|^{2}}d\lambda
\]
\[
\leq\left(\frac{1}{2n}\right)\frac{1}{(n-1)^{2}}\int_{F_{r}}\left|\nabla\left(\frac{fe^{\frac{-g(N)}{2}}}{N^{\frac{1}{2}}}\right)\right|^{2}d\lambda+\left(\frac{1}{2n}\right)\frac{1}{n-1}\int_{\partial F_{r}}\frac{f^{2}e^{-g(N)}}{N|x|^{2}}\sum_{j=1}^{2n}\frac{x_{j}<X_{j}I,\nabla_{euc}\left(\alpha\frac{|x|^{2}g'(N)}{N}\right)>}{\left|\nabla_{euc}\left(\alpha\frac{|x|^{2}g'(N)}{N}\right)\right|}dH^{2n}
\]
\[
\leq\left(\frac{1}{2n}\right)\frac{1}{(n-1)^{2}}\int_{F_{2}}\left|\nabla\left(\frac{fe^{\frac{-g(N)}{2}}}{N^{\frac{1}{2}}}\right)\right|^{2}d\lambda+\left(\frac{1}{2n}\right)\frac{1}{n-1}\int_{\partial F_{r}}\frac{f^{2}e^{-g(N)}}{N|x|^{2}}\sum_{j=1}^{2n}\frac{x_{j}<X_{j}I,\nabla_{euc}\left(\alpha\frac{|x|^{2}g'(N)}{N}\right)>}{\left|\nabla_{euc}\left(\alpha\frac{|x|^{2}g'(N)}{N}\right)\right|}dH^{2n}
\]
Integrating both sides of the inequality from $r=1$ to $r=2,$ we
get:
\[
\left(\frac{1}{2n}\right)\int_{1}^{2}\int_{F_{1}}\frac{\vert fe^{\frac{-g(N)}{2}}\vert^{2}}{N|x|^{2}}d\lambda dr
\]
\[
\leq\left(\frac{1}{2n}\right)\frac{1}{(n-1)^{2}}\int_{1}^{2}\int_{F_{2}}\left|\nabla\left(\frac{fe^{\frac{-g(N)}{2}}}{N^{\frac{1}{2}}}\right)\right|^{2}d\lambda dr
\]
\[
+\left(\frac{1}{2n}\right)\frac{1}{n-1}\int_{1}^{2}\int_{\partial F_{r}}\frac{f^{2}e^{-g(N)}}{N|x|^{2}}\sum_{j=1}^{2n}\frac{x_{j}<X_{j}I,\nabla_{euc}\left(\alpha\frac{|x|^{2}g'(N)}{N}\right)>}{\left|\nabla_{euc}\left(\alpha\frac{|x|^{2}g'(N)}{N}\right)\right|}dH^{2n}dr
\]
To recover the full measure in the boundary term, we use the Coarea
formula: 
\[
\left(\frac{1}{2n}\right)\int_{F_{1}}\frac{\vert fe^{\frac{-g(N)}{2}}\vert^{2}}{N|x|^{2}}d\lambda\leq
\]
\begin{equation}
\frac{1}{2n(n-1)^{2}}\int_{F_{2}}\left|\nabla\left(\frac{fe^{\frac{-g(N)}{2}}}{N^{\frac{1}{2}}}\right)\right|^{2}d\lambda+\frac{1}{2n(n-1)}\int_{\{1<\alpha\frac{|x|^{2}g'(N)}{N}<2\}}\frac{f^{2}e^{-g(N)}}{N|x|^{2}}\sum_{j=1}^{2n}x_{j}<X_{j}I,\nabla_{euc}\left(\alpha\frac{|x|^{2}g'(N)}{N}\right)>d\lambda\label{eq:20-1}
\end{equation}
It remains to compute the right hand side of (\ref{eq:20-1}). The
first term,
\[
A=\left(\frac{1}{2n}\right)\frac{1}{(n-1)^{2}}\int_{F_{2}}\left|\nabla\left(\frac{fe^{\frac{-g(N)}{2}}}{N^{\frac{1}{2}}}\right)\right|^{2}d\lambda
\]
\[
=\left(\frac{1}{2n}\right)\frac{1}{(n-1)^{2}}\int_{F_{2}}\left|\nabla f\frac{e^{-\frac{g(N)}{2}}}{N^{\frac{1}{2}}}-f\frac{g'(N)}{2}\frac{\nabla Ne^{-\frac{g(N)}{2}}}{N^{\frac{1}{2}}}-\frac{1}{2}\frac{f\nabla Ne^{-\frac{g(N)}{2}}}{N^{\frac{3}{2}}}\right|^{2}d\lambda
\]
\[
=\left(\frac{1}{2n}\right)\frac{1}{(n-1)^{2}}\int_{F_{2}}\frac{|\nabla f|^{2}}{N}e^{-g(N)}d\lambda+\left(\frac{1}{8n}\right)\frac{1}{(n-1)^{2}}\int_{F_{2}}f^{2}g'(N)^{2}\frac{|\nabla N|^{2}}{N}e^{-g(N)}d\lambda
\]
\[
+\left(\frac{1}{8n}\right)\frac{1}{(n-1)^{2}}\int_{F_{2}}f^{2}\frac{|\nabla N|^{2}}{N^{3}}e^{-g(N)}d\lambda-\frac{1}{2n(n-1)^{2}}\int_{F_{2}}f\nabla f\cdot\triangledown N\frac{g'(N)}{N}e^{-g(N)}d\lambda
\]
\[
-\frac{1}{2n(n-1)^{2}}\int_{F_{2}}f\nabla f\cdot\triangledown N\frac{1}{N^{2}}e^{-g(N)}d\lambda+\frac{1}{4n(n-1)^{2}}\int_{F_{2}}f^{2}g'(N)\frac{|\nabla N|^{2}}{N^{2}}e^{-g(N)}d\lambda.
\]
Using (\ref{eq:24-2}), 
\[
A\leq\left(\frac{1}{2n}\right)\frac{1}{(n-1)^{2}}\int_{F_{2}}\frac{|\nabla f|^{2}}{N}e^{-g(N)}d\lambda+\frac{(2n+1)^{2}}{2^{6}n^{3}(n-1)^{2}}\int_{F_{2}}f^{2}g'(N)^{2}\frac{|x|^{2}}{N^{3}}e^{-g(N)}d\lambda
\]
\[
+\frac{(2n+1)^{2}}{2^{6}n^{3}(n-1)^{2}}\int_{F_{2}}f^{2}\frac{|x|^{2}}{N^{5}}e^{-g(N)}d\lambda+\frac{(2n+1)}{2^{\frac{5}{2}}n^{2}(n-1)^{2}}\int_{F_{2}}|f||\nabla f|\frac{g'(N)|x|}{N^{2}}e^{-g(N)}d\lambda
\]
\[
+\frac{(2n+1)}{2^{\frac{5}{2}}n^{2}(n-1)^{2}}\int_{F_{2}}|f||\nabla f|\frac{|x|}{N^{3}}e^{-g(N)}d\lambda+\frac{(2n+1)^{2}}{2^{5}n^{3}(n-1)^{2}}\int_{F_{2}}f^{2}g'(N)\frac{|x|^{2}}{N^{4}}e^{-g(N)}d\lambda.
\]
\\
Using the fact that $F_{2}=\left\{ \alpha\frac{|x|^{2}g'(N)}{N}<2\right\} $
and that $|x|\leq C_{n}N,$
\[
A\leq\left(\frac{1}{2n}\right)\frac{1}{(n-1)^{2}}\int_{F_{2}}\frac{|\nabla f|^{2}}{N}e^{-g(N)}d\lambda+\frac{(2n+1)^{2}}{\alpha2^{5}n^{3}(n-1)^{2}}\int_{F_{2}}f^{2}\frac{g'(N)}{N^{2}}e^{-g(N)}d\lambda
\]
\[
+\frac{(2n+1)^{2}C_{n}^{2}}{2^{6}n^{3}(n-1)^{2}}\int_{F_{2}}f^{2}\frac{1}{N^{3}}e^{-g(N)}d\lambda+\frac{(2n+1)C_{n}}{2^{\frac{5}{2}}n^{2}(n-1)^{2}}\int_{F_{2}}|f||\nabla f|\frac{g'(N)}{N}e^{-g(N)}d\lambda
\]
\[
+\frac{(2n+1)C_{n}}{2^{\frac{5}{2}}n^{2}(n-1)^{2}}\int_{F_{2}}|f||\nabla f|\frac{1}{N^{2}}e^{-g(N)}d\lambda+\frac{(2n+1)^{2}}{\alpha2^{4}n^{3}(n-1)^{2}}\int_{F_{2}}f^{2}\frac{1}{N^{3}}e^{-g(N)}d\lambda.
\]
\\
Using Cauchy's inequality with $\gamma$: $ab\leq\gamma a^{2}+{\displaystyle \frac{b^{2}}{4\gamma}}$
on ${\displaystyle \frac{(2n+1)C_{n}}{2^{\frac{5}{2}}n^{2}(n-1)^{2}}\int_{F_{2}}|f||\nabla f|\frac{g'(N)}{N}e^{-g(N)}d\lambda}$
with 
\\
\\
${\displaystyle a=\frac{|f|g'(N)}{N}}$ and $b={\displaystyle \frac{(2n+1)C_{n}}{2^{\frac{5}{2}}n^{2}(n-1)^{2}}|\nabla f|}$
and Cauchy's inequality with $\gamma$ on ${\displaystyle \frac{(2n+1)C_{n}}{2^{\frac{5}{2}}n^{2}(n-1)^{2}}\int_{F_{2}}|f||\nabla f|\frac{1}{N^{2}}e^{-g(N)}d\lambda}$
\\
\\
with ${\displaystyle a=\frac{|f|}{N^{2}}}$ and $b={\displaystyle \frac{(2n+1)C_{n}}{2^{\frac{5}{2}}n^{2}(n-1)^{2}}|\nabla f|}$
in addition to $N>1,$
\[
A\leq\left(\left(\frac{1}{2n}\right)\frac{1}{(n-1)^{2}}+\frac{C_{n}^{2}(2n+1)^{2}}{\gamma2^{6}n^{4}(n-1)^{4}}\right)\int_{F_{2}}|\nabla f|^{2}e^{-g(N)}d\lambda
\]
\[
+\left(\frac{C_{n}^{2}(2n+1)^{2}}{2^{6}n^{3}(n-1)^{2}}+\frac{(2n+1)^{2}}{\alpha2^{4}n^{3}(n-1)^{2}}+\gamma\right)\int_{F_{2}}f^{2}e^{-g(N)}d\lambda+\left(\frac{(2n+1)^{2}}{\alpha2^{5}n^{3}(n-1)^{2}}+\gamma\right)\int_{F_{2}}f^{2}\frac{g'(N)}{N^{2}}e^{-g(N)}d\lambda.
\]
Hence, 
\begin{equation}
A\leq\left(\frac{(2n+1)^{2}}{\alpha2^{5}n^{3}(n-1)^{2}}+\gamma\right)\int_{F_{2}}f^{2}\frac{g'(N)}{N^{2}}e^{-g(N)}d\lambda+C\int_{F_{2}}|\nabla f|^{2}e^{-g(N)}d\lambda+D\int_{F_{2}}f^{2}e^{-g(N)}d\lambda.\label{eq:25}
\end{equation}
For the second term of (\ref{eq:20-1}), 
\[
B=\left(\frac{1}{2n}\right)\frac{1}{n-1}\int_{\{1<\alpha\frac{|x|^{2}g'(N)}{N}<2\}}\frac{f^{2}e^{-g(N)}}{N|x|^{2}}\sum_{j=1}^{2n}x_{j}<X_{j}I,\nabla_{euc}\left(\alpha\frac{|x|^{2}g'(N)}{N}\right)>d\lambda.
\]
\\
On the anisotropic Heisenberg group, for $e_{i}$ the standard Euclidean
basis on $\mathbb{R}^{2n+1},$
\\
\[
X_{j}I\cdot e_{i}=\begin{cases}
0\;\;\;\;\;\;\;\;\;\;\;\;\;\;\;\;\;\;\;\;\;\;\; & for\;\;\;i\neq j\;\;\;and\;\;\;i\leq2n\\
1\;\;\;\;\;\;\;\;\;\;\;\;\;\;\;\;\;\;\;\;\;\;\; & for\;\;\;i=j\;\;\;and\;\;\;i\leq2n\\
\sum_{l=1}^{2n}\Lambda_{jl}x_{l}\;\;\;\;\;\;\; & for\;\;\;i=2n+1,
\end{cases}
\]
where 
\[
\Lambda_{jl}=\begin{cases}
-\frac{1}{2} & j=1,\ \ l=n+1\\
\frac{1}{2} & j=n+1,\ \ l=1\\
-1 & j=2,3,...,n,\ \ l=k+n\\
1 & j=n+2,n+3,...,2n,\ \ l=k-n\\
0 & otherwise
\end{cases}
\]
Also,
\[
\nabla_{euc}\left(\alpha\frac{|x|^{2}g'(N)}{N}\right)\cdot e_{i}=\begin{cases}
{\displaystyle 2\alpha\frac{x_{i}g'(N)}{N}+\frac{\alpha|x|^{2}g''(N)\partial_{x_{i}}N}{N}-\frac{\alpha|x|^{2}g'(N)\partial_{x_{i}}N}{N^{2}}}\;\;\;\;\;\;\;\;\;\;\;\; & for\;\;\;i=j\;\;\;and\;\;\;i\leq2n\\
\\
{\displaystyle \frac{\alpha|x|^{2}g''(N)\partial_{t}N}{N}-\frac{\alpha|x|^{2}g'(N)\partial_{t}N}{N^{2}}}\;\;\;\;\;\;\;\;\;\;\;\;\;\;\;\;\;\;\;\;\;\;\;\;\;\;\;\; & for\;\;\;i=2n+1.
\end{cases}
\]
Taking the dot product and summing, 
\[
{\displaystyle \sum_{j=1}^{2n}x_{j}<X_{j}I,\nabla_{euc}\left(\alpha\frac{|x|^{2}g'(N)}{N}\right)>\;=2\alpha\frac{|x|^{2}g'(N)}{N}+\sum_{j=1}^{2n}x_{j}\frac{\alpha|x|^{2}g''(N)\partial_{x_{i}}N}{N}-\sum_{j=1}^{2n}x_{j}\frac{\alpha|x|^{2}g'(N)\partial_{x_{i}}N}{N^{2}}}
\]
\[
+\sum_{j=1}^{2n}x_{j}\left(\frac{\alpha|x|^{2}g''(N)\partial_{t}N}{N}-\frac{\alpha|x|^{2}g'(N)\partial_{t}N}{N^{2}}\right)\sum_{l=1}^{2n}\Lambda_{jl}x_{l}
\]
\[
=2\alpha\frac{|x|^{2}g'(N)}{N}+\frac{\alpha|x|^{2}g''(N)x\cdot\triangledown N}{N}-\frac{\alpha|x|^{2}g'(N)x\cdot\triangledown N}{N^{2}}+\left(\frac{\alpha|x|^{2}g''(N)\partial_{t}N}{N}-\frac{\alpha|x|^{2}g'(N)\partial_{t}N}{N^{2}}\right)\sum_{j=1}^{2n}\sum_{l=1}^{2n}\Lambda_{jl}x_{l}x_{j}
\]
\[
=2\alpha\frac{|x|^{2}g'(N)}{N}+\frac{\alpha|x|^{2}g''(N)x\cdot\triangledown N}{N}-\frac{\alpha|x|^{2}g'(N)x\cdot\triangledown N}{N^{2}},
\]
where ${\displaystyle \sum_{j=1}^{2n}\sum_{l=1}^{2n}\Lambda_{jl}x_{l}x_{j}=0}$
since ${\displaystyle \Lambda}$ is skew symmetric. Hence, using (\ref{eq:24-2}),
\[
\left|\sum_{j=1}^{2n}x_{j}<X_{j}I,\nabla_{euc}\left(\alpha\frac{|x|^{2}g'(N)}{N}\right)>\right|\leq2\alpha\frac{|x|^{2}g'(N)}{N}+\frac{\alpha|x|^{3}g''(N)|\triangledown N|}{N}+\frac{\alpha|x|^{3}g'(N)|\triangledown N|}{N^{2}}
\]
\[
\leq2\alpha\frac{|x|^{2}g'(N)}{N}+\frac{\alpha(2n+1)|x|^{4}g''(N)}{2^{\frac{3}{2}}nN^{2}}+\frac{\alpha(2n+1)|x|^{4}g'(N)}{2^{\frac{3}{2}}nN^{3}}.
\]
Therefore, replacing, 
\[
B=\left(\frac{1}{2n}\right)\frac{1}{n-1}\int_{\{1<\alpha\frac{|x|^{2}g'(N)}{N}<2\}}\frac{f^{2}e^{-g(N)}}{N|x|^{2}}\sum_{j=1}^{2n}x_{j}<X_{j}I,\nabla_{euc}\left(\alpha\frac{|x|^{2}g'(N)}{N}\right)>d\lambda
\]
\[
\leq\left(\frac{1}{n}\right)\frac{\alpha}{n-1}\int_{\{1<\alpha\frac{|x|^{2}g'(N)}{N}<2\}}\frac{f^{2}g'(N)}{N^{2}}e^{-g(N)}d\lambda+\frac{\alpha(2n+1)}{2^{\frac{5}{2}}n^{2}(n-1)}\int_{\{1<\alpha\frac{|x|^{2}g'(N)}{N}<2\}}\frac{f^{2}|x|^{2}g''(N)}{N^{3}}e^{-g(N)}d\lambda
\]
\[
+\frac{\alpha(2n+1)}{2^{\frac{5}{2}}n^{2}(n-1)}\int_{\{1<\alpha\frac{|x|^{2}g'(N)}{N}<2\}}\frac{f^{2}|x|^{2}g'(N)}{N^{4}}e^{-g(N)}d\lambda.
\]
Using the fact that we are integrating over $\{1<\alpha\frac{|x|^{2}g'(N)}{N}<2\},$
\[
B\leq\left(\frac{1}{n}\right)\frac{\alpha}{n-1}\int_{\{1<\alpha\frac{|x|^{2}g'(N)}{N}<2\}}\frac{f^{2}g'(N)}{N^{2}}e^{-g(N)}d\lambda+\frac{(2n+1)}{2^{\frac{3}{2}}n^{2}(n-1)}\int_{\{1<\alpha\frac{|x|^{2}g'(N)}{N}<2\}}\frac{f^{2}g''(N)}{N^{2}g'(N)}e^{-g(N)}d\lambda
\]
\begin{equation}
+\frac{(2n+1)}{2^{\frac{3}{2}}n^{2}(n-1)}\int_{\{1<\alpha\frac{|x|^{2}g'(N)}{N}<2\}}\frac{f^{2}}{N^{3}}e^{-g(N)}d\lambda.\label{eq:26}
\end{equation}
Using the condition of the theorem that $g''(N)\leq g'(N)^{2}$ on
$\{N\geq1\},$ we bound the second term on the right hand side of
(\ref{eq:26}). Now we go back to (\ref{eq:24-1}):
\[
\frac{1}{2^{5+\frac{2}{n}}}\int f^{2}\left[\frac{g^{'}\left(N\right)}{N^{2}}-\left(Q-1\right)\right]e^{-g\left(N\right)}d\lambda
\]
\[
\leq\beta\int\frac{f^{2}}{N^{2}\vert x\vert^{2}}e^{-g\left(N\right)}d\lambda+\frac{1}{2n}\int\frac{f^{2}}{N\vert x\vert^{2}}e^{-g\left(N\right)}d\lambda+\frac{(2n+1)^{2}}{8n^{2}\beta}\int|\triangledown f|^{2}e^{-g\left(N\right)}d\lambda.
\]
From (\ref{eq:19-1}) we have that 
\[
\frac{1}{2n}\int\frac{f^{2}}{N\vert x\vert^{2}}e^{-g\left(N\right)}d\lambda\leq\frac{1}{2n}\int_{F}\frac{f^{2}}{N\vert x\vert^{2}}e^{-g\left(N\right)}d\lambda+\frac{\alpha}{2n}\int_{E}\frac{f^{2}g'(N)}{N^{2}}e^{-g\left(N\right)}d\lambda.
\]
From (\ref{eq:25}) and (\ref{eq:26}) we have that 
\[
\frac{1}{2n}\int_{F}\frac{f^{2}}{N\vert x\vert^{2}}e^{-g\left(N\right)}d\lambda
\]
\[
\leq\left(\frac{(2n+1)^{2}}{\alpha2^{5}n^{3}(n-1)^{2}}+\gamma+\frac{\alpha}{n(n-1)}\right)\int_{F_{2}}f^{2}\frac{g'(N)}{N^{2}}e^{-g(N)}d\lambda+C\int_{F_{2}}|\nabla f|^{2}e^{-g(N)}d\lambda+D\int_{F_{2}}f^{2}e^{-g(N)}d\lambda.
\]
We combine the last three inequalities and repeat the same procedure
to $\beta\int\frac{f^{2}}{N^{2}\vert x\vert^{2}}e^{-g\left(N\right)}d\lambda$
to get:
\begin{equation}
\left(\frac{1}{2^{5+\frac{2}{n}}}-\frac{\alpha}{2n}-\frac{(2n+1)^{2}}{\alpha2^{5}n^{3}(n-1)^{2}}-\frac{\alpha}{n(n-1)}-\gamma-\beta\right)\int f^{2}\frac{g^{'}\left(N\right)}{N^{2}}e^{-g\left(N\right)}d\lambda\label{eq:26-1}
\end{equation}
\[
\leq C\int_{F_{2}}|\nabla f|^{2}e^{-g(N)}d\lambda+D\int_{F_{2}}f^{2}e^{-g(N)}d\lambda.
\]
To get the U-Bound, we need the left hand side of (\ref{eq:26-1})
to be positive i.e. and find suitable $\alpha,n,$ $\gamma,$ and
$\beta.$ Since $\gamma$ and $\beta$ can be chosen to be arbitrarily
small, we need to find solutions to the following inequality:
\begin{equation}
\frac{1}{2^{5+\frac{2}{n}}}-\frac{\alpha}{2n}-\frac{(2n+1)^{2}}{\alpha2^{5}n^{3}(n-1)^{2}}-\frac{\alpha}{n(n-1)}>0.\label{eq:27}
\end{equation}
First, we determine $\alpha:$ Let $f(\alpha)={\displaystyle \frac{1}{2^{5+\frac{2}{n}}}-\frac{\alpha}{2n}-\frac{(2n+1)^{2}}{\alpha2^{5}n^{3}(n-1)^{2}}-\frac{\alpha}{n(n-1)}.}$
\[
f'(\alpha)=-\frac{1}{2n}+\frac{(2n+1)^{2}}{\alpha^{2}2^{5}n^{3}(n-1)^{2}}-\frac{1}{n(n-1)}.
\]
Solving for $f'(\alpha)=0,$
\[
\alpha={\displaystyle \frac{2n+1}{2^{2}n(n^{2}-1)^{\frac{1}{2}}}}.
\]
Replacing $\alpha$ in (\ref{eq:27}),
\[
\frac{1}{2^{5+\frac{2}{n}}}-\frac{(2n+1)}{2^{3}n^{2}(n^{2}-1)^{\frac{1}{2}}}-\frac{(n^{2}-1)^{\frac{1}{2}}(2n+1)}{2^{3}n^{2}(n-1)^{2}}-\frac{(2n+1)}{2^{2}n^{2}(n-1)(n^{2}-1)^{\frac{1}{2}}}>0
\]
\[
1>2^{3+\frac{2}{n}}\left(\frac{2n+1}{n^{2}}\right)\left(\frac{(n+1)^{\frac{1}{2}}}{(n-1)^{\frac{3}{2}}}\right).
\]
This holds true if we choose $n>5.$ Hence, we get:
\[
\int f^{2}\left(\frac{g^{'}\left(N\right)}{N^{2}}\right)e^{-g\left(N\right)}d\lambda\leq C\int|\nabla f|^{2}e^{-g(N)}d\lambda+D\int f^{2}e^{-g(N)}d\lambda.
\]
Second, for $q>2,$replacing $|f|$ by ${\displaystyle \vert f\vert^{\frac{q}{2}},}$we
get:
\begin{equation}
\int\frac{g'\left(N\right)}{N^{2}}\vert f\vert^{q}d\mu\leq C\int\left|\triangledown\vert f\vert^{\frac{q}{2}}\right|^{2}d\mu+D\int\vert f\vert^{q}d\mu.\label{eq:21-1}
\end{equation}
Calculating,
\[
\int\left|\triangledown\vert f\vert^{\frac{q}{2}}\right|^{2}d\mu=\int\left|\frac{q}{2}\vert f\vert^{\frac{q-2}{2}}\left(sgn\left(f\right)\right)\triangledown f\right|^{2}d\mu
\]
\[
\leq\int\frac{q^{2}}{4}\vert f\vert^{q-2}\vert\triangledown f\vert^{2}d\mu.
\]
Using Hölder\textquoteright s inequality, 
\[
\int\left|\triangledown\vert f\vert^{\frac{q}{2}}\right|^{2}d\mu\leq\int\frac{q^{2}}{4}\vert f\vert^{q-2}\vert\triangledown f\vert^{2}d\mu~\leq\frac{q^{2}}{4}\left(\int\vert f\vert^{q}d\mu\right)^{\frac{q-2}{q}}\left(\int\vert\triangledown f\vert^{q}d\mu\right)^{\frac{2}{q}}
\]
\begin{equation}
\leq\frac{q\left(q-2\right)}{4}\int\vert f\vert^{q}d\mu+\frac{q}{2}\int\vert\triangledown f\vert^{q}d\mu.\label{eq:22-1}
\end{equation}
Where the last inequality uses $ab{\displaystyle \leq\frac{a^{p^{'}}}{p^{'}}+\frac{b^{q'}}{q'},}$
with ${\displaystyle a=\left(\int\vert f\vert^{q}d\mu\right)^{\frac{q-2}{q}},}$
${\displaystyle b=\left(\int\vert\triangledown f\vert^{q}d\mu\right)^{\frac{2}{q}},}$
and $p'$ and $q'$ are conjugates. Choosing ${\displaystyle p^{'}=\frac{q}{q-2}},$
we obtain ${\displaystyle \frac{1}{q'}=1-\frac{q-2}{q}}$ , so ${\displaystyle q^{'}=\frac{q}{2}}.$
Using the inequalities (\ref{eq:21-1}) and (\ref{eq:22-1}), we get,
\[
\int\frac{g'\left(N\right)}{N^{2}}\vert f\vert^{q}d\mu\leq C\int\left|\triangledown\vert f\vert^{\frac{q}{2}}\right|^{2}d\mu+D\int\vert f\vert^{q}d\mu
\]
\[
\leq C'\int\vert\triangledown f\vert^{q}d\mu+D'\int\vert f\vert^{q}d\mu.
\]
\end{proof}

\section{$q-$Poincaré and $\beta-$Logarithmic Sobolev Inequalities\protect 
}We now have the U-Bound (\ref{eq:u}) at our disposal and are ready
to prove the q-Poincaré inequality using the U-Bound method of \cite{key-33}:

Let $\lambda$ be a measure satisfying the q-Poincaré inequality for
every ball ${\displaystyle B_{R}=\{x:N\left(x\right)<R\},~}$ i.e.
there exists a constant $C_{R}\in\left(0,\infty\right)$ such that
\[
\frac{1}{\vert B_{R}\vert}\int_{B_{R}}\left|f-\frac{1}{\vert B_{R}\vert}\int_{B_{R}}f\right|^{q}d\lambda\leq C_{R}\frac{1}{\vert B_{R}\vert}\int_{B_{R}}\vert\triangledown f\vert^{q}d\lambda,
\]

where $1\leq q<\infty.$ Note that we have this Poincaré inequality
on balls in the setting of Nilpotent lie groups thanks to J. Jerison's
celebrated paper \cite{key-1}. Later on, we apply the following result
of
\begin{thm}[Hebisch, Zegarli\'{n}ski \cite{key-33}]
 Let $\mu$ be a probability measure on $\mathbb{R}^{m}$ which is
absolutely continuous with respect to the measure $\lambda$ and such
that
\[
\int f^{q}\eta d\mu\leq C\int\vert\triangledown f\vert^{q}d\mu+D\int f^{q}d\mu
\]
with some non-negative function $\eta$ and some constants $C,D\in\left(0,\infty\right)$
independent of a function $f.$ If for any $L\in\left(0,\infty\right)$
there is a constant $A_{L}$ such that ${\displaystyle \frac{1}{A_{L}}\leq\frac{d\mu}{d\lambda}\leq A_{L}}$
on the set $\left\{ \eta<L\right\} $ and, for some $R\in\left(0,\infty\right)$
(depending on L), we have $\left\{ \eta<L\right\} \subset B_{R},$
then $\mu$ satisfies the q-Poincaré inequality
\[
\mu\vert f-\mu f\vert^{q}\leq c\mu\vert\triangledown f\vert^{q}
\]
with some constant $c\in\left(0,\infty\right)$ independent of $f.$
\end{thm}

The role of $\eta$ in Theorem 4 is played by ${\displaystyle \frac{g'(N)}{N^{2}}}$
from the U-Bound of Theorem 3. Hence, we get the following corollaries: 
\begin{cor}
The Poincaré inequality for $q\geq2$ holds for the measure ${\displaystyle d\mu=\frac{exp\left(-cosh\left(N^{k}\right)\right)}{Z}d\lambda}$,
where $\lambda$ is the Lebesgue measure, and k$\geq1$ in the setting
of the anisotropic Heisenberg group $\mathbb{R}^{2n+1}$ with $n>5$. 
\end{cor}

\begin{proof}
$g\left(N\right)=cosh\left(N^{k}\right),$ so $g'\left(N\right)=kN^{k-1}sinh\left(N^{k}\right),$
and\\
 ${\displaystyle g''(N)=k(k-1)N^{k-2}sinh(N^{k})+k^{2}N^{2k-2}cosh(N^{k})}.$
First ${\displaystyle g''(N)\leq k^{2}N^{2k-2}sinh^{2}(N^{k})=g'(N)^{2}},$
on the set $\{N>{\displaystyle \frac{3}{2}\},}$so the condition of
Theorem 2 is satisfied. Secondly,
\[
\int\frac{g'\left(N\right)}{N^{2}}f^{q}d\mu=\int f^{q}\left[kN^{k-3}sinh\left(N^{k}\right)\right]d\mu
\]
\[
=\int_{\left\{ N<\frac{3}{2}\right\} }f^{q}\left[kN^{k-3}\frac{e^{N^{k}}-e^{-N^{k}}}{2}\right]d\mu+\int_{\left\{ N\geq\frac{3}{2}\right\} }f^{q}\left[kN^{k-3}\frac{e^{N^{k}}-e^{-N^{k}}}{2}\right]d\mu
\]
\[
\leq\int_{\left\{ N<\frac{3}{2}\right\} }f^{q}\left[\left(\frac{3}{2}\right)^{k-3}k\frac{e^{\left(\frac{3}{2}\right)^{k}}-e^{-\left(\frac{3}{2}\right)^{k}}}{2}\right]d\mu+C^{'}\int_{\left\{ N\geq1\right\} }\vert\triangledown f\vert^{q}d\mu+D^{'}\int_{\left\{ N\geq1\right\} }\vert f\vert^{q}d\mu
\]
\[
\leq C\int\vert\triangledown f\vert^{q}d\mu+D\int\vert f\vert^{q}d\mu.
\]
Thus, the conditions of Theorem 4 are satisfied for $\eta=kN^{k-3}sinh\left(N^{k}\right),$
and $k\geq1.$ So, the Poincaré inequality holds for $q\geq2.$
\end{proof}
\begin{cor}
The Poincaré inequality for $q\geq1$ holds for the measure ${\displaystyle d\mu=\frac{exp\left(-N^{k}\right)}{Z}d\lambda,}$
where $\lambda$ is the Lebesgue measure, and $k\geq4$ in the setting
of the anisotropic Heisenberg group $\mathbb{R}^{2n+1}$ with $n>5$. 
\end{cor}

\begin{proof}
Let $g(N)=N^{k},$so $g^{'}\left(N\right)=kN^{k-1},$ and $g''(N)=k(k-1)N^{k-2}.$
First ${\displaystyle g''(N)\leq k^{2}N^{2k-2}=g'(N)^{2},}$ so the
condition of Theorem 2 is satisfied. Second, 
\[
\int\frac{g'\left(N\right)}{N^{2}}f^{q}d\mu=\int f^{q}\left[kN^{k-3}\right]d\mu=\int_{\left\{ N<1\right\} }f^{q}\left[kN^{k-3}\right]d\mu+\int_{\left\{ N\geq1\right\} }f^{q}\left[kN^{k-3}\right]d\mu
\]
\[
\leq\int_{\left\{ N<1\right\} }kf^{q}d\mu+C^{'}\int_{\left\{ N\geq1\right\} }\vert\triangledown f\vert^{q}d\mu+D^{'}\int_{\left\{ N\geq1\right\} }\vert f\vert^{q}d\mu
\]
\[
\leq C\int\vert\triangledown f\vert^{q}d\mu+D\int\vert f\vert^{q}d\mu.
\]
Thus, the conditions of Theorem 4 are satisfied for $\eta=kN^{k-3},$
and $k\geq4.$ So, the Poincaré inequality holds for $q\ensuremath{\geq2.}$
\end{proof}
The following corollary improves Corollary 6 in an interesting way.
Namely, at a cost of a logarithmic factor, we now get the Poincaré
inequality for polynomial growth of order $k\geq3.$ 
\begin{cor}
The Poincaré inequality for $q\geq2$ holds for the measure ${\displaystyle d\mu=\frac{exp\left(-N^{k}log\left(N+1\right)\right)}{Z}d\lambda},$
where $\lambda$ is the Lebesgue measure, and $k\geq3$ in the setting
of the anisotropic Heisenberg group $\mathbb{R}^{2n+1}$ with $n>5$. 
\end{cor}

\begin{proof}
Let $g(N)=N^{k}log\left(N+1\right),$ so ${\displaystyle g^{'}\left(N\right)=kN^{k-1}log(N+1)+\frac{N^{k}}{N+1},}$
and
\[
{\displaystyle g''(N)=k(k-1)N^{k-2}log(N+1)+\frac{2kN^{k-1}}{N+1}-\frac{N^{k}}{(N+1)^{2}}.}
\]
First ${\displaystyle g''(N)\leq k^{2}N^{2k-2}(log(N+1))^{2}+\frac{N^{2k}}{(N+1)^{2}}+2kN^{k-1}log(N+1)\frac{N^{k-1}}{N+1}=g'(N)^{2},}$
so the condition of Theorem 2 is satisfied. Secondly,
\[
\int\frac{g'\left(N\right)}{N^{2}}f^{q}d\mu=\int f^{q}\left[kN^{k-3}log(N+1)+\frac{N^{k-2}}{N+1}\right]d\mu
\]
\[
=\int_{\left\{ N<1\right\} }f^{q}\left[kN^{k-3}log(N+1)+\frac{N^{k-2}}{N+1}\right]d\mu+\int_{\left\{ N\geq1\right\} }f^{q}\left[kN^{k-3}log(N+1)+\frac{N^{k-2}}{N+1}\right]d\mu
\]
\[
\leq\int_{\left\{ N<1\right\} }\left(klog(2)+1\right)f^{q}d\mu+C^{'}\int_{\left\{ N\geq1\right\} }\vert\triangledown f\vert^{q}d\mu+D^{'}\int_{\left\{ N\geq1\right\} }\vert f\vert^{q}d\mu
\]
\[
\leq C\int\vert\triangledown f\vert^{q}d\mu+D\int\vert f\vert^{q}d\mu.
\]
Thus, the conditions of Theorem 4 are satisfied for ${\displaystyle \eta=kN^{k-3}log(N+1)+\frac{N^{k-2}}{N+1},}$
and $k\geq3.$ So, the Poincaré inequality holds for $q\geq1.$ 
\end{proof}
To get the $\beta-$Logarithmic Sobolev inequality, we use the following
theorem by the authors of this paper in \cite{key-36} (which generalises
J.Inglis et al.'s Theorem 2.1 \cite{key-2}).
\begin{thm}
Let U be a locally lipschitz function on $\mathbb{R}^{N}$which is
bounded below such that $Z=\int e^{-U}d\lambda<\infty,$ and ${\displaystyle d\mu=\frac{e^{-U}}{Z}d\lambda.}$
Let $\phi:[0,\infty)\rightarrow\mathbb{R}^{+}$ be a non-negative,
non-decreasing, concave function such that $\phi(0)>0,$ and $\phi'(0)>0.$
Assume the following classical Sobolev inequality is satisfied:
\[
\left(\int|f|^{q+\epsilon}d\lambda\right)^{\frac{q}{q+\epsilon}}\leq a\int|\triangledown f|^{q}d\lambda+b\int|f|^{q}d\lambda
\]
 for some $a,$ $b\in[0,\infty),$ and $\epsilon>0.$ Moreover, if
for some $A,$ $B\in[0,\infty),$ we have:
\[
\mu\left(|f|^{q}(\phi(U)+|\triangledown U|^{q})\right)\leq A\mu|\triangledown f|^{q}+B\mu|f|^{q},
\]
Then, there exists constants $C,$ $D\in[0,\infty)$ such that:
\[
\mu\left(|f|^{q}\phi\left(\left|log\frac{|f|^{q}}{\mu|f|^{q}}\right|\right)\right)\leq C\mu|\triangledown f|^{q}+D\mu|f|^{q},
\]
 for all locally Lipschitz functions $f.$ 
\end{thm}

Theorem 10 of \cite{key-36} proves that for ${\displaystyle d\mu=\frac{e^{-\alpha N^{p}}}{Z}d\lambda,}$
where $\alpha>0,\;p\geq1,$ $N$ is a smooth homogeneous norm, and
$Z$ is the normalization constant, the measure $\mu$ satisfies no
$\beta$-logarithmic Sobolev inequality $(0<\beta\leq1)$ for $1<q<{\displaystyle \frac{2p\beta}{p-1}}.$
However for $p\geq4,$ ${\displaystyle 0<\beta\leq\frac{p-3}{p},}$
and $q\geq2,$ we will show that $\mu$ satisfies $\beta$-logarithmic
Sobolev inequality.

We will use the following theorem: 
\begin{thm}
Let $\mathbb{R}^{2n+1}$ be an anisotropic Heisenberg group with $n>5.$
Let the probability measure be ${\displaystyle d\mu=\frac{e^{-g(N)}}{Z}d\lambda,}$
where $Z$ is the normalization constant and $N^{-2n}$ the fundamental
solution. Let $g:\left[0,\infty\right)\rightarrow\left[0,\infty\right)$
be a differentiable increasing function such that $g'(N)$ is increasing,
${\displaystyle g(N)\leq\left(c\frac{g'(N)}{N^{2}}\right)^{\frac{1}{\beta}}},$
and $g''(N)<dg'(N)^{2}$ on $\{N\geq1\},$ for some constants $c$
and $d.$ Then 
\[
\mu\left(|f|^{q}\left|log\left(\frac{|f|^{q}}{\mu|f|^{q}}\right)\right|^{\beta}\right)\leq C\mu|f|^{q}+D\mu|\triangledown f|^{q},
\]

for $C$ and $D$ positive constants and for $q\geq2.$
\end{thm}

The proof follows closely that of Theorem 11 in \cite{key-36}, with
some modification.
\begin{proof}
Choose $\phi(x)=(1+x)^{\beta},$ which satisfies the conditions of
Theorem 7. On $\{N\geq1\},$ $g''(N)<dg'(N)^{2}\leq g'(N)^{3}N^{3},$
so the condition of Theorem 3 is satisfied. Thus, on $\{N\geq1\},$
we have the U-bound (\ref{eq:u}):
\[
\mu\left(\frac{g'(N)}{N^{2}}|f|^{q}\right)\leq C\mu|\triangledown f|^{q}+D\mu|f|^{q}.
\]
By the condition ${\displaystyle g(N)\leq\left(c\frac{g'(N)}{N^{2}}\right)^{\frac{1}{\beta}},}$
we obtain ${\displaystyle \phi(g(N))=(1+g(N))^{\beta}\leq\frac{g'(N)}{N^{2}}}$
on $\{N\geq1\}$. Hence, since $g(N)$ is increasing and using the
U-bound,
\begin{equation}
\begin{array}{cl}
{\displaystyle \mu\left(\phi(g(N))|f|^{q}\right)} & {\displaystyle \leq\int_{\{N\geq1\}}\left(\frac{g'(N)}{N^{2}}|f|^{q}\right)d\mu+\int_{\{N<1\}}\phi(g(N))|f|^{q}d\mu}\\
\\
 & \leq{\displaystyle \int_{\{N\geq1\}}\left(\frac{g'(N)}{N^{2}}|f|^{q}\right)d\mu+\int_{\{N<1\}}(1+g(1))^{\beta}|f|^{q}d\mu}\\
\\
 & {\displaystyle \leq C\mu|\triangledown f|^{q}+D\mu|f|^{q}.}
\end{array}\label{eq:19-1-1}
\end{equation}

In order to use Theorem 8, it remains to prove:
\begin{equation}
\mu\left(|f|^{q}|\triangledown g(N)|^{q}\right)\leq C\mu|f|^{q}+D\mu|\triangledown f|^{q}.\label{eq:20-1-1}
\end{equation}

On $\{N<1\},$ since $g'(N)$ is increasing and using (\ref{eq:24-2}),\\
 $\int_{\{N<1\}}\left(|f|^{q}|\triangledown g(N)|^{q}\right)d\mu=\int_{\{N<1\}}\left(|f|^{q}|g'(N)\triangledown N|^{q}\right)d\mu\leq{\displaystyle \frac{(2n+1)^{q}}{2^{\frac{3q}{2}}n^{q}}}\int_{\{N<1\}}|f|^{q}|g'(1)|^{q}d\mu.$
\\
We now need to consider $\{N\geq1\}:$
\begin{equation}
\int|f|^{q}\left(\triangledown g(N)\cdot V-\triangledown\cdot V\right)d\mu=\int\triangledown|f|^{q}\cdot Vd\mu\leq\frac{\epsilon}{p}\int|f|^{q}|V|^{p}d\mu+\frac{1}{\epsilon^{\frac{q}{p}}}q^{q-1}\int|\triangledown f|^{q}d\mu,\label{eq:21-1-1}
\end{equation}
where the last inequality uses $\text{\ensuremath{{\displaystyle ab \leq\epsilon\frac{a^{p}}{p} + \frac{b^{q}}{\epsilon^{\frac{q}{p}}q} ,}}}$
where $a=|f|^{q-1}|V|,$ and $b=q|\triangledown f|.$ Let ${\displaystyle V=\triangledown N\frac{|x|^{q-2}}{N^{q-2}}g'(N)^{q-1}.}$
Since $\triangledown g(N)=g'(N)\triangledown N,$ then $\triangledown g(N)\cdot V=|\triangledown g(N)|^{q},$
which is the term on the left hand side of (\ref{eq:20-1-1}). Using
the inequality (\ref{eq:24-2}) on the first term on the right hand
side of (\ref{eq:21-1-1}) we get
\[
\begin{array}{cl}
{\displaystyle \frac{\epsilon}{p}\int|f|^{q}|V|^{p}d\mu} & {\displaystyle =\frac{\epsilon}{p}\int|\triangledown N|^{p}|f|^{q}\frac{|x|^{(q-2)p}g'(N)^{q}}{N^{(q-2)p}}d\mu}\\
\\
 & \leq{\displaystyle \frac{\epsilon}{p}\frac{(2n+1)^{q}}{2^{\frac{3q}{2}}n^{q}}\int|f|^{q}\frac{|x|^{q}g'(N)^{q}}{N^{q}}d\mu}
\end{array}
\]
which can subtracted from the left hand side of (\ref{eq:21-1-1})
since by choosing $\epsilon$ small enough and noting that using (\ref{eq:20-2}),
\[
\triangledown g(N)\cdot V=|\triangledown N|^{2}\frac{|x|^{q-2}}{N^{q-2}}g'(N)^{q}\geq\frac{1}{2^{5+\frac{2}{n}}}\frac{|x|^{q}}{N^{q}}g'(N)^{q}.
\]

It remains to compute $\triangledown\cdot V.$ Using the fact that
$N^{-2n}$ is the fundamental solution, ${\displaystyle \triangle N=\frac{(Q-1)|\triangledown N|^{2}}{N},}$
$\;Q=2n+2$ the homogeneous dimension and ${\displaystyle |\triangledown N|^{2}\leq\frac{(2n+1)^{2}|x|^{2}}{2^{3}n^{2}N^{2}}}$
(\ref{eq:24-2}),
\[
\triangledown\cdot V=\Delta N\frac{|x|^{q-2}}{N^{q-2}}g'(N)^{q-1}+(q-2)\frac{|x|^{q-4}g'(N)^{q-1}x\cdot\triangledown N}{N^{q-2}}-(q-2)\frac{|\triangledown N|^{2}|x|^{q-2}g'(N)^{q-1}}{N^{q-1}}
\]
\[
+(q-1)g'(N)^{q-2}g''(N)\frac{|x|^{q-2}|\triangledown N|^{2}}{N^{q-2}}.
\]
\[
|\triangledown\cdot V|\leq\frac{(Q-1)(2n+1)^{2}}{2^{3}n^{2}}\frac{|x|^{q}}{N^{q+1}}g'(N)^{q-1}+(q-2)\frac{(2n+1)|x|^{q-2}g'(N)^{q-1}}{2^{\frac{3}{2}}nN^{q-1}}
\]
\[
+(q-2)\frac{(2n+1)^{2}}{2^{3}n^{2}}\frac{|x|^{q}g'(N)^{q-1}}{N^{q+1}}+(q-1)\frac{(2n+1)^{2}}{2^{3}n^{2}}g'(N)^{q-2}g''(N)\frac{|x|^{q}}{N^{q}}.
\]
All terms can be absorbed by the first term in (\ref{eq:21-1-1}).
Using (\ref{eq:19-1-1}) and (\ref{eq:20-1-1}), the condition of
Theorem 8 is satisfied, and we obtain $\beta-$logarithmic Sobolev
inequality:
\[
\mu\left(|f|^{q}\left|log\left(\frac{|f|^{q}}{\mu|f|^{q}}\right)\right|^{\beta}\right)\leq C\mu|f|^{q}+D\mu|\triangledown f|^{q}
\]
 for $C$ and $D$ positive constants independent of $f$. 
\end{proof}
As a corollary, we obtain the following result:
\begin{cor}
Let $\mathbb{R}^{2n+1}$ be an anisotropic Heisenberg group with $n>5$
and $N$ the homogeneous norm corresponding to the fundamental solution
of the sub-Laplacian equation. Let the probability measure be ${\displaystyle d\mu=\frac{e^{-\beta N^{p}}}{Z}d\lambda,}$
where $Z$ is the normalization constant. Then, for $p\geq4$ and
${\displaystyle 0<\beta\leq\frac{p-3}{p},}$

\[
\mu\left(|f|^{q}\left|log\left(\frac{|f|^{q}}{\mu|f|^{q}}\right)\right|^{\beta}\right)\leq C\mu|f|^{q}+D\mu|\triangledown f|^{q},
\]
 for $C$ and $D$ positive constants and for $q\geq2.$ 
\end{cor}


\begin{thebibliography}{10}
\bibitem{key-25}D. Bakry, F. Baudoin, M. Bonnefont, D. Chafai. \textit{On
gradient bounds for the heat kernel on the Heisenberg group.} J. Funct.
Anal. 255, 1905--1938 (2008).

\bibitem{key-3}Z. Balogh and J. Tyson. \textit{Polar Coordinates
in Carnot Groups.} Math Z. 2002, 241:4, 697--730.

\bibitem{key-20}R. Beals, B. Gaveau, and P. Greiner. \textit{The
Green function of model step two hypoelliptic operators and the analysis
of certain tangential Cauchy Riemann complexes.} Adv. Math. 121 (1996),
288--345.

\bibitem{key-21}T. Bieske. \textit{On the Lie Algebra of polarizable
Carnot groups.} Anal. Math. Phys. 10, 80 (2020). 

\bibitem{key-38}S. Bobkov and B. Zegarli\'{n}ski. \textit{Entropy
bounds and isoperimetry.} Mem. Amer. Math.Soc., 176(829), 2005. 

\bibitem{key-41}Th. Bodineau and B. Helffer. \textit{On Log-Sobolev
inequalities for unbounded spin systems.} J. Funct. Anal. 166 (1999),
168-178. 

\bibitem{key-11-1}A. Bonfiglioli, E. Lanconelli, and F. Uguzzoni.
\textit{Stratified Lie Groups and Potential Theory for their Sub-Laplacians.}
Springer Monographs in Mathematics. Springer, 2007. 

\bibitem{key-36}E. Bou Dagher and B. Zegarli\'{n}ski. \textit{Coercive
Inequalities and U-Bounds. } arXiv:2105.01759 [math.FA]. 

\bibitem{key-8-1}E. Bou Dagher and B. Zegarli\'{n}ski. \textit{Coercive
Inequalities on Carnot Groups: Taming Singularities.} Preprint 2021.

\bibitem{key-19}M. Brakalova, I. Markina and A. Vasil'ev. \textit{Modules
of systems of measures on polarizable Carnot groups.} Ark. Mat., 54
(2016), 371-401.

\bibitem{key-35}M. Chatzakou, S. Federico, B. Zegarlinski. \textit{q-Poincaré
inequalities on Carnot Groups with a filiform Lie algebra.} arXiv:2007.04689v2
{[}math.FA{]}. 

\bibitem{key-32}W.S. Cohn, G. Lu, and P. Wang. \textit{Sub-elliptic
global high order Poincaré inequalities in stratified Lie groups and
applications.} (English summary) J. Funct. Anal. 249 (2007), no. 2,
393--424.

\bibitem{key-9}L. D\textquoteright Ambrosio. \textit{Hardy type inequalities
related to degenerate elliptic differential operators.} Ann. Sc. Norm.
Super. Pisa Cl. Sci. (5), 4(3):451--486, 2005.

\bibitem{key-13}D. Danielli, N. Garofalo, and N.C. Phuc. \textit{Hardy--Sobolev
type inequalities with sharp constants in Carnot--Carathéodory spaces.}
Potential Anal., 34:223--242, 2011.

\bibitem{key-7}G.B. Folland. \textit{Subelliptic estimates and function
spaces on nilpotent Lie groups.} Ark. Mat., 13(2):161--207, 1975.

\bibitem{key-10}J.A. Goldstein and I. Kombe. \textit{The Hardy inequality
and nonlinear parabolic equations on Carnot groups. }Nonlinear Anal.,
69(12):4643-- 4653, 2008.

\bibitem{key-14}J.A. Goldstein, I. Kombe, and A. Yener. \textit{A
unified approach to weighted Hardy type inequalities on Carnot groups.
}Discrete Contin. Dyn. Syst., 37(4):2009--2021, 2017.

\bibitem{key-37}A. Guionnet and B. Zegarli\'{n}ski. \textit{Lectures
on logarithmic Sobolev inequalities.} Séminaire de Probabilités, XXXVI,
1-134, Lecture Notes in Math., 1801, Springer, Berlin, 2003.

\bibitem{key-33}W. Hebisch and B. Zegarli\'{n}ski. \textit{Coercive
inequalities on metric measure spaces.} J. Funct. Anal., 258:814--851,
2010. 

\bibitem{key-34}J. Inglis. \textit{Coercive Inequalities for Generators
of Hörmander Type.} Doctor of Philosophy of the University of London
and the Diploma of Imperial College, Department of Mathematics Imperial
College, 2010.

\bibitem{key-2}J. Inglis, V. Kontis, B. Zegarli\'{n}ski. \textit{From
U-Bounds to Isoperimetry with Applications.} J. Funct. Anal., 260
(2011) 76-116.

\bibitem{key-1}J. Jerison. \textit{The Poincaré inequality for vector
fields satisfying Hörmander's condition.} Duke Math. J. 53 (1986),
no. 2, 503-523. 

\bibitem{key-15}I. Kombe. \textit{Sharp weighted Rellich and uncertainty
principle inequalities on Carnot groups.} Commun. Appl. Anal., 14(2):251--271,
2010.

\bibitem{key-23}H.Q. Li. \textit{Estimation optimale du gradient
du semi-groupe de la chaleur sur le groupe de Heisenberg.} J. Funct.
Anal. 236, 369--394 (2006).

\bibitem{key-27}X. Li, C.Z. Lu, and H.L. Tang. \textit{Poincaré inequalities
for vector fields satisfying Hörmander\textquoteright s condition
in variable exponent Sobolev spaces.} Acta Math. Sin. (Engl. Ser.)
31 (2015), no. 7, 1067--1085.

\bibitem{key-28}G. Lu. \textit{Local and global interpolation inequalities
on the Folland-Stein Sobolev spaces and polynomials on stratified
groups.} (English summary) Math. Res. Lett. 4 (1997), no. 6, 777--790.

\bibitem{key-29}G. Lu. \textit{Polynomials, higher order Sobolev
extension theorems and interpolation inequalities on weighted Folland-Stein
spaces on stratified groups.} (English summary) Acta Math. Sin.(Engl.
Ser.) 16 (2000), no. 3, 405--444.

\bibitem{key-30}G. Lu and R.L. Wheeden. \textit{High order representation
formulas and embedding theorems on stratified groups and generalizations.}
Studia Math. 142 (2000), no. 2, 101--133. (Reviewer: G. B. Folland).

\bibitem{key-31}G. Lu and R.L. Wheeden, R. \textit{Simultaneous representation
and approximation for- mula and high-order Sobolev embedding theorems
on stratified groups.} (English summary) Constr. Approx. 20 (2004),
no. 4, 647--668.

\bibitem{key-40}P. \L ugiewicz and B. Zegarli\'{n}ski. \textit{Coercive
Inequalities for Hörmander Type Generators in Infinite Dimensions.
}J. Funct. Anal. 247 (2007), 438-476.

\bibitem{key-26}T. Mechler. \textit{Hypoelliptic heat kernel inequalities
on Lie groups. }Stochastic Process. Appl. 118 (2008), no.3, 368-388
(Reviewer: T. Coulhon).

\bibitem{key-39}C. Roberto and B. Zegarli\'{n}ski. \textit{Orlicz-Sobolev
inequalities for sub-Gaussian measures and ergodicity of Markov semi-groups.}
J. Func. Anal. 243. (2006) 28-66.

\bibitem{key-3-1-1}M. Ruzhansky and N. Yessirkegenov. \textit{Factorization
and Hardy-Rellich Inequalities on Stratified Groups.} arXiv:1706.05108
(2017).

\bibitem{key-8}L. Saloff-Coste. \textit{Théorèmes de Sobolev et inégalités
de Trudinger sur certain groupes de Lie.} C. R. Acad. Sci. Paris 306
(1988), 305--308.

\bibitem{key-22}N.T. Varopoulos, L. Saloff-Coste and T. Coulhon.\textit{
Analysis and Geometry on Groups.} Cambridge Tracts in Mathematics,
100, Cambridge University Press, Cambridge (1992).

\bibitem{key-16}J. Wang and P. Niu. \textit{Sharp weighted Hardy
type inequalities and Hardy--Sobolev type inequalities on polarizable
Carnot groups. }C. R. Math. Acad. Sci. Paris Ser. I, 346:1231--1234,
2008.

\bibitem{key-17}Q. Yang. \textit{Best constants in the Hardy-Rellich
type inequalities on the Heisenberg group. }J. Math. Anal. Appl.,
342, 423-431 (2008).

\bibitem{key-42}N. Yosida. \textit{The log-Sobolev inequality for
weakly coupled lattice fields. }Probab. Theor. Relat. Field 115 (1999)
1-40. 
\end{thebibliography}
\end{document}